\documentclass{amsart}

\usepackage{graphicx} 
\usepackage{amssymb} 
\usepackage{amsmath}
\usepackage{faktor}
\usepackage{amsthm}
\usepackage{cancel}
\usepackage{esvect}
\usepackage[english]{babel}

\usepackage{tikz}             
\usepackage{tikz-cd}                                                                        
\pgfdeclarelayer{nodelayer}
\pgfdeclarelayer{edgelayer}
\pgfsetlayers{edgelayer,nodelayer,main}

\tikzset{new style 0/.style={circle, draw=black!20, fill=black!3, very thick, minimum size=1mm}}
\tikzset{simple/.style={thin}}

\usepackage{graphicx}
\graphicspath{ {./images/} }
\usepackage{color}

\newtheorem{theorem}{Theorem}[section]
\newtheorem{definition}[theorem]{Definition}
\newtheorem{corollary}{Corollary}[theorem]
\newtheorem{proposition}[theorem]{Proposition}
\newtheorem{lemma}[theorem]{Lemma}
\newtheorem{remark}[theorem]{Remark}
\newtheorem{example}[theorem]{Example}
\title{Modular Symbols over Function Fields of Elliptic Curves}
\author{Seong Eun Jung }
\date{}
\begin{document}
\begin{sloppypar}
\maketitle
    \begin{abstract}
        Let $k =\mathbb{F}_q$ be the finite field of $q$ elements and $E$ an elliptic curve over $k$. Let $F = k(E)$ be the function field over $E$ and let $\mathcal{O} = k[E]$ be the ring of integers. We fix the place at $\infty$ of $F$ and let $F_{\infty}$ be the completion. The group $\overline\Gamma = {\rm{GL}}_2(\mathcal{O})$ acts on $\mathcal{T}$, the Bruhat-Tits building of ${\rm{PGL}}_2(F_{\infty})$. In this article, we construct the group of modular symbols over $\Gamma$, a congruence subgroup of $\overline\Gamma$. We prove that this space is given by an explicit set of generators and relations among them.
    \end{abstract}

\section{Introduction}
Modular symbols have been defined over various settings. We begin with mdoular symbols over $\mathbb{Q}$, also known as \textit{classical modular symbols}. We extend the complex upper half plane $\mathcal{H}$ to $\mathcal{H^*}:= \mathcal{H} \cup \mathbb{P}^1(\mathbb{Q})$. We call any point in $\mathbb{P}^1(\mathbb{Q})$ a cusp. The modular group ${\rm{SL}_2}(\mathbb{Z})$, or any of its congruence subgroups, which we call $\Gamma$, acts on $\mathcal{H^*}$. The quotient space is the modular curve $X := \Gamma\backslash\mathcal{H^*}$. The geodesic from a cusp $\alpha$ to a cusp $\beta$ on $\mathcal{H^*}$ projects to an oriented path on $X$. These geodesics determine a class in $H_1(X, \text{cusps}; \mathbb{Z})$, the relative first homology group of $X$ with respect to its cusps. A modular symbol is a class in this homology group. (See \cite{manin} for more details).

Over the rational function field $F = \mathbb{F}_q(T)$, the analog to $\mathcal{H}$ is the Bruhat-Tits building of ${\rm{PGL}_2}(F_{\infty})$, where $F_{\infty}$ is the completion of $F$ with respect to the place at $\infty$. This building is the $(q+1)$-regular infinite tree, which we denote $\mathcal{T}$. Fix a vertex $v$ on $\mathcal{T}$. Any infinite non-backtracking path starting at $v$ is called an \textit{end}, which corresponds to $\mathbb{P}^1(F_{\infty})$. The ends corresponding to 
$\alpha \in \mathbb{P}^1(F)$ are called \textit{rational ends}, and we call $\alpha$ a \textit{cusp}. There is a unique non-backtracking path on $\mathcal{T}$ from a cusp $\alpha$ to another cusp $\beta$. This is the analog to a geodesic over $\mathcal{H^*}$. Let $\mathcal{O} = \mathbb{F}_q[T]$ and let $\Gamma$ be a congruence subgroup of ${\rm{GL}_2}(\mathcal{O})$. The group $\Gamma$ acts on $\mathcal{T}$. The path from $\alpha$ to $\beta$ determines a class in $H_1(\Gamma\backslash\mathcal{T}, \text{cusps};  \mathbb{Z})$, which are the modular symbols in this setting. (See \cite{teitelbaum} for more details).  

In both cases, Manin and Teitelbaum showed that the group of modular symbols is given by a set of generators and relations among them that is finite modulo the action of $\Gamma$. This allows for explicit computation of the homology groups. Another application is the computation of Hecke operators. There is a subgroup of the classical symbols that is dual to $\mathbb{S}_2(\Gamma)$, the group of weight two cuspforms over $\Gamma$. There is an action of the Hecke operators on the modular symbols that is dual to the action of the cuspforms, and this action can be represented as finite matrices, allowing us to compute the Hecke eigenvalues and eigenvectors. The eigenvectors are cuspforms corresponding to certain abelian varieties. 

In this article, we give an analogous result for modular symbols over the function field of an elliptic curve. Given an elliptic curve $E$ over $k = \mathbb{F}_q$, let $\mathcal{O} = k[E]$ be the affine coordinate ring and $F = k(F)$ its field of fractions. Again, we fix the place at $\infty$ and let $F_{\infty}$ be the completion of $F$ with respect to this place. The analog of $\mathcal{H}$ is the Bruhat-Tits building over ${\rm{PGL}_2}(F_{\infty})$. This is again the $(q+1)$-regular infinite tree, which we call $\mathcal{T}$. The discrete group $\overline\Gamma = {\rm{GL}}_2(\mathcal{O})$ acts on the tree $\mathcal{T}$. Takahashi \cite{takahashi} gives the construction of the quotient space $\overline\Gamma\backslash\mathcal{T}$, which can be identified with a subtree of $\mathcal{T}$ that we call $\mathcal{S}$. 

Like the rational function field case, any end corresponds to an element in $\mathbb{P}^1(F)$, which we call a cusp. For any two cusps $\alpha, \beta$, there is again a unique non-backtracking path from $\alpha$ to $\beta$. Any congruence subgroup $\Gamma \subseteq \overline\Gamma$ acts on $\mathcal{T}$, meaning we could define a modular symbol in this setting to be a class in $H_1(\Gamma\backslash\mathcal{T}, \text{cusps};  \mathbb{Z})$. Kondo and Yasuda \cite{kondo} provide a similar definition for modular symbols over elliptic function fields.

\begin{definition}\textup{\cite[Definition 4.4]{kondo}}
The group of modular symbols $MS(\Gamma)_{\mathbb{Z}}$ is the submodule inside $H^{BM}_1(\Gamma/\mathcal{T}, \mathbb{Z})$, the first Borel-Moore homology group, generated by the classes of the ordered pair of cusps $\alpha, \beta \in \mathbb{P}^1(F)$. \end{definition}
According to \cite[Theorem 4.5.1]{kondo}, we have:
\begin{equation}
H^{BM}_1(\Gamma/\mathcal{T}, \mathbb{Z}) = MS(\Gamma)_{\mathbb{Z}} \text{ and } H^{BM}_1(\Gamma/\mathcal{T}, \mathbb{Q}) = MS(\Gamma)_{\mathbb{Z}} \otimes \mathbb{Q}.
\end{equation}

To achieve the goal of this article, we will use the following definition. Let the group of modular symbols, denoted $\mathbb{M}_2$, be the $\mathbb{Q}$-vector space generated by ordered pairs of cusps modulo certain relations as Stein \cite{stein} did. The group $\Gamma$ acts on the pairs, and taking $\mathbb{M}_2$ modulo the coinvariants gives us a group, which we denote as $\mathbb{M}_2(\Gamma)$, that serves as an analog to Kondo and Yasuda's group $MS(\Gamma)_{\mathbb{Z}} \otimes \mathbb{Q}$. The goal of this article can be reformulated as providing a presentation of $\mathbb{M}_2(\Gamma)$ using a set of generators and relations between them that are finite modulo the action of $\Gamma$. 

We will show that there are four types of generators, which we call \textit{reduced symbols}. After defining the reduced symbols, we will prove the following result.
\begin{theorem}[Theorem \ref{uni sum}]\label{main 1}
    Any modular symbol can be written as a finite sum of the reduced symbols.
\end{theorem}

To find the relations, we construct a CW complex $\mathcal{K}$ using the action of $\overline\Gamma$ on $\mathcal{T}$ where the vertices correspond to the cusps and the 1-cells to the reduced symbols. There is a $\Gamma$-action on the complex. Using \cite[Chapter VII]{brown}, we construct a spectral sequence over $\mathcal{K}$ and compute the equivariant relative homology groups, which allows us to explicitly define $\mathbb{M}_2(\Gamma)$. This gives us the following result to be made explicit.

\begin{theorem}[Theorem \ref{main result}]\label{main 2}
    The group of modular symbols $\mathbb{M}_2(\Gamma)$ is given by a set of generators and relations between them that are finite modulo the action of $\Gamma$.
\end{theorem}

The complex $\mathcal{K}$ has the advantage of capturing all relations between the reduced symbols, which $\mathcal{T}$ does not. This comes at the expense of not making any claims about either $H_1(\Gamma\backslash\mathcal{T}, \text{cusps};  \mathbb{Z})$ or $H^{BM}_1(\Gamma/\mathcal{T}, \mathbb{Z})$. However, we will show that we are able to compute the Steinberg homology of $\Gamma$ (to be defined in Definition \ref{def: steinberg hom q}) and that a modular symbol determines a class in this group.

The motivation of proving Theorem \ref{main 2} is to be able to explicitly compute modular symbols over congruence subgroups such as $\Gamma_0(\mathfrak{p})$, the group of elements in $\rm{{GL}_2}(\mathcal{O})$ congruent to an upper triangular matrix modulo $\mathfrak{p}$, a prime ideal in $\mathcal{O}$. Theorem \ref{main 2} also allows us to compute the Steinberg homology of $\Gamma_0(\mathfrak{p})$ and the action of the Hecke operators. The eigenvalues of the Hecke operators will give information about certain automorphic forms over ${\rm{GL}_2}(F)$ as defined in \cite[Chapter 5]{kondo}. 

The article is structured as follows. In Section \ref{sec: modular symbols}, we establish notation and label all the vertices of $\mathcal{T}$. We define modular symbols, the reduced symbols, and the Steinberg homology. We show any modular symbol can be written as a sum of the reduced symbols (Theorem \ref{main 1}). In Section \ref{sec: relations among symbols}, we construct $\mathcal{K}$, the spectral sequence, and the equivariant relative homology group. This allows us to write the relations among the reduced symbols and prove Theorem \ref{main 2}. Finally, in Section 4, we finish a proof needed to show that we indeed compute the Steinberg homology.

\section{Modular Symbols}\label{sec: modular symbols}
\subsection{The Tree $\mathcal{T}$}
Let $f(x,y) = y^2 + a_1xy + a_3y - x^3 - a_2x^2 - a_4x - a_6 = 0$ define a nonsingular elliptic curve $E$ with $a_i \in k = \mathbb{F}_q$. Let $k[E]$ be the affine coordinate ring and $k(E)$ its field of fractions. Let $t = x/y$ be a local uniformizer at the point at $\infty$. Let $F_{\infty}$ be the completion of $F = k(E)$ at $\infty$. This allows us to embed $\mathcal{O} = k[E]$ into $F_{\infty} = k((t))$ such that $\text{ord}(x) = -2$ and $\text{ord}(y) = -3$. Let $\overline\Gamma = {\rm{GL}}_2(\mathcal{O})$, $\Gamma$ a congruence subgroup of $\overline\Gamma$, and let $\mathcal{O}_{\infty} = k[[t]]$ be the valuation ring. Furthermore, let $K = {\rm{GL}}_2(\mathcal{O_{\infty}})$ and let $G = {\rm{GL}}_2(F_{\infty})$. 

Now $G/KZ$ can be identified with the vertices of the tree $\mathcal{T}$. A vertex corresponds to a coset $gKZ$ and has $q+1$ adjacent vertices (see Figure \ref{tree}). There is an edge between two vertices $g_1KZ, g_2KZ$ if $g_1^{-1}g_2 = 
\begin{pmatrix}
    t & b \\
    0 & 1
\end{pmatrix}
\text{or} 
\begin{pmatrix}
    t^{-1} & 0 \\
    0 & 1
\end{pmatrix}
$ modulo $KZ$ with $b \in k$. 

\begin{figure}[htb]
\centering
\includegraphics[width=4cm, height=4cm]{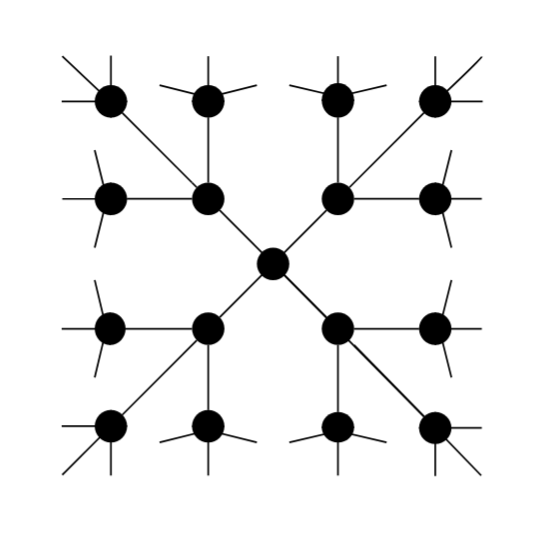}
\caption{The Tree $\mathcal{T}$ for $q=3$.}
\label{tree}
\end{figure}

This is the Bruhat-Tits building for ${\rm{PGL}_2}(F_{\infty})$. As stated earlier, it is the $(q+1)$-regular infinite tree and plays a similar role as $\mathcal{H}$. Furthermore, it is a simplicial complex. Fixing a vertex $v$, any infinite non-backtracking path starting at $v$ is an \textit{end}, which correspond to $\mathbb{P}^1(F_{\infty})$, and the ends corresponding to $\mathbb{P}^1(F)$ is a \textit{rational end}. Now $\overline\Gamma$ acts on $\mathcal{T}$, giving us the quotient space $\overline\Gamma\backslash\mathcal{T}$. Takahashi computes the quotient space using the stabilizers of the vertices and shows that $\overline\Gamma\backslash\mathcal{T}$ can be identified with a subtree $\mathcal{S}$ of $\mathcal{T}$. 

\subsection{The Subtree $\mathcal{S}$}
We briefly describe the construction of $\mathcal{S}$. (See \cite[Section 2]{takahashi} for more details). Let $f(x,y) = 0$ define an elliptic curve $E$ over $k$. We let $o$ denote the vertex on $\mathcal{T}$ corresponding to the coset 
$
\begin{pmatrix}
    t & t^{-1} \\
    0 & 1
\end{pmatrix}
KZ$. We also call this vertex type $o$. The adjacent vertices are given the following labels: $v(l)$ for all $l \in \mathbb{P}^1(k)$. The adjacent vertices of the $v(l)$ depend on the value of $l$. 

First, if $l = \infty$ or $x = l \in \mathbb{F}_q$ yields a single solution $y = m$ in the equation $f(x,y)=0$, there is an end with the labels $c(p,n)$ for all $n \in \mathbb{N}_{\geq 1}$ starting at $v(l)$ where $p = \infty$ or $p = (l,m)$. At $n = 1$, there is one other adjacent vertex labeled $e(p)$, which we call type $e$. We call the vertex $v(l)$ type $os$ since there is one solution to $f$. 

If $x = l' \in \mathbb{F}_q$ yields no solution of $y$ to $f(x,y) = 0$, then $v(l')$ has no adjacent vertices other than $o$. We call such $v(l')$ vertices type $ns$ since there is no solution to $f$.

Finally, if $x = l \in \mathbb{F}_q$ yields two solutions of $y = m$ and $y = m'$ under $f(x,y)=0$, then we have two ends with labels $c(p,n)$ and $c(p',n)$ for all $n \in \mathbb{N}_{\geq 1}$ adjacent to $v(l)$ where $p = (l,m)$ and $p=(l'm)$. We call such $v(x)$ vertices type $s$ for square. This is because if $a_1 = a_3 = 0$, then $f(x,y) = y^2 - g(x)$ and $g(l)$ is a square in $\mathbb{F}_q$.

\begin{figure}[htb]
\centering
\includegraphics[width=8cm, height=7cm]{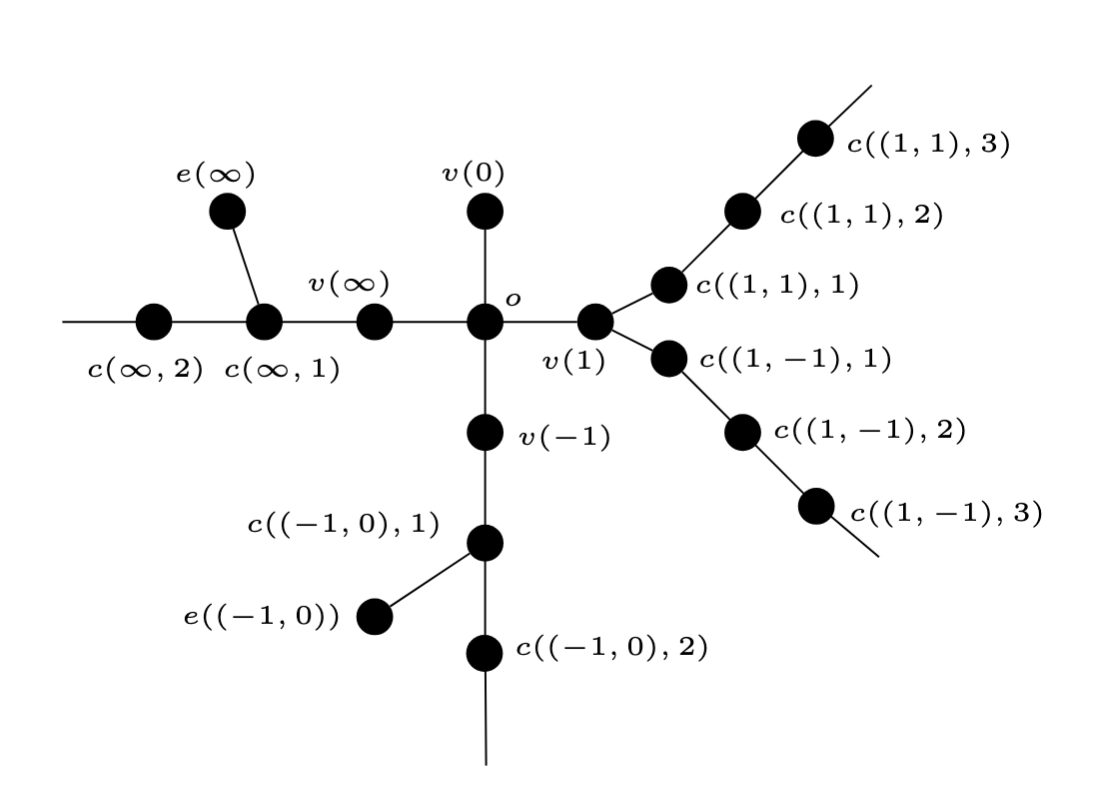}
\caption{The Subtree $\mathcal{S}$ for $y^2=x^3+x-1$ over $\mathbb{F}_3$.}
\label{fig: quot_tree}
\end{figure}

The main result of \cite{takahashi} is the following.
\begin{theorem}\textup{\cite[Theorem 3]{takahashi}}
    For any elliptic curve $E$ over $k$, we have $\mathcal{S} \cong \overline\Gamma\backslash\mathcal{T}$.
\end{theorem}

An infinite path with labels $c(p,n)_{n \geq 1}$ is called a \textit{rational end}. Such paths are in bijection with $E(\mathbb{F}_q)$, the set of rational points on $E$ according to \cite[pg 86-87]{takahashi}. We will prove that the rational ends correspond to a \textit{cusp}, which corresponds to a point in $\mathbb{P}^1(F)$.
\begin{lemma}\label{lem: end cusp s}
A rational end on $\mathcal{S}$ goes to a point $\alpha$ in $\mathbb{P}^1(F)$, which we call a cusp.
\end{lemma}

\begin{proof}
Let $p = (l,m)$ be a rational point on $E$ that is not the point $\infty$, and let $c(p,n)_{n \geq 1}$ denote the rational end corresponding with $p$. We know that $\alpha = (y-m)/(x-l) \in k(E)$ has a Laurent series representation $k(t) \in k_{\infty}$. Let $k_n(t)$ be the Laurent polynomial of degree $n-1$.
Using Proposition 1 and the description of $c(p,n)$ given on page 89 of \cite{takahashi}, the vertex $c(p,n)$ can be represented by the coset $\begin{pmatrix}
t^n & k_n(t) \\
0 & 1
\end{pmatrix}KZ$. We follow the notation of \cite{takahashi} and use $\phi(t^n, k_n(t))$. Every vertex can be written so that the coset representative has the vector $(0, 1)$ for the bottom row. Taking the limit as $n$ goes to $\infty$ gives us $\phi(0, k(t))$. So the rational end goes to the cusp $\alpha$ represented by the Laurent series $k(t)$ that lies on the boundary of $\mathcal{T}$.
The vertex $c(\infty, n)$ can be represented by $\phi(t^{-n}, 0)$. Taking the limit as $n$ goes to $\infty$, this rational end must go to the cusp $\infty$. 
\end{proof}

\begin{definition}
Let $R$ denote the set $\{\infty, (y-m)/(x-l) \in \mathbb{P}^1(F) : (l,m) \in E(\mathbb{F}_q)\}$. We call the elements in $R$ \textbf{orbit-representative cusps}.
\end{definition}
\begin{remark}
\normalfont The rational ends corresponding to these cusps on $\mathcal{S}$ are not $\overline\Gamma$-equivalent. By Lemma \ref{lem: end cusp s}, it follows that the cusps are also not $\overline\Gamma$-equivalent.
\end{remark}
\subsection{Labeling the Vertices of $\mathcal{T}$}
As in the case of the rational function field \cite[Example 2.4.1]{serre}, we extend the labeling on the vertices of $\mathcal{S}$ to the vertices of
$\mathcal{T}$. This is important because the labels will allow us to identify modular symbols. Let $pr: \mathcal{T} \to \overline\Gamma\backslash\mathcal{T} \cong \mathcal{S}$ be the canonical quotient map. Then the label of $v \in \mathcal{T}$ is defined to be the label of $pr(v)$. We discuss this in further detail using results from \cite{takahashi}.

Define a successor relation on the vertices of $\mathcal{S}$ in the following way. If we have an end $v_0, \dots , v_n, \dots, v_m, \dots $ with $v_0 = o, v_n = v, v_m = w$ for $0\leq n<m$, then we call $w$ the successor of $v$ and $v$ the predecessor of $w$. We use the following result. 
\begin{theorem}\label{stab_in_quot_tree}\textup{\cite[Theorem 4]{takahashi}} 
Given a vertex $v$, denote $\overline\Gamma_v \subset \overline\Gamma$ as the stabilizer of $v$. Every vertex $v$ in $\mathcal{S}$ satisfies the following two properties: 
\begin{itemize}
    \item For any successor $w$ of $v$ in $\mathcal{S}$, we have $\overline\Gamma_v \subset \overline\Gamma_w$. 
    \item For any successor $w$ of $v$ in $\mathcal{T}$ but not in $\mathcal{S}$, $w$ is $\overline\Gamma_v$-equivalent to the predecessor of $v$ in $\mathcal{S}$.
\end{itemize}
\end{theorem}

Theorem \ref{stab_in_quot_tree} tells us that we can label all the vertices in $\mathcal{T}$ in the following way. For any vertex $v$, the adjacent vertices that were not labeled on $\mathcal{S}$ are given the same label as the predecessor of $v$. The adjacent vertices of $o$ are given by $v(l)$. For $v(l)$, all remaining adjacent vertices are labeled $o$. For $c((l,m),1)$ where $x = l$ yields one solution $y = m$ to $f(x,y) = 0$, the adjacent vertices are $e(p)$, $c(p,2)$ where $p = (l,m)$ and $q-1$ vertices labeled $v(l)$. When $l$ has two solutions $m, m'$ to $f(x,y) = 0$, the adjacent vertices are $c(p,2)$, $c(p',2)$ where $p = (l, m)$ and $p' = (l,m')$ and $q-1$ vertices labeled $v(l)$. For any $c(p,n)$ for $n \geq 2$, we have one vertex labeled $c(p,n+1)$ and the remaining are $c(p,n-1)$. For $e(p)$, all adjacent vertices are labeled $c(p,1)$. (See Figure \ref{fig: tree all labels} for an example).

\begin{figure}[htb]
\centering
\includegraphics[width=10.5cm, height=7.5cm]{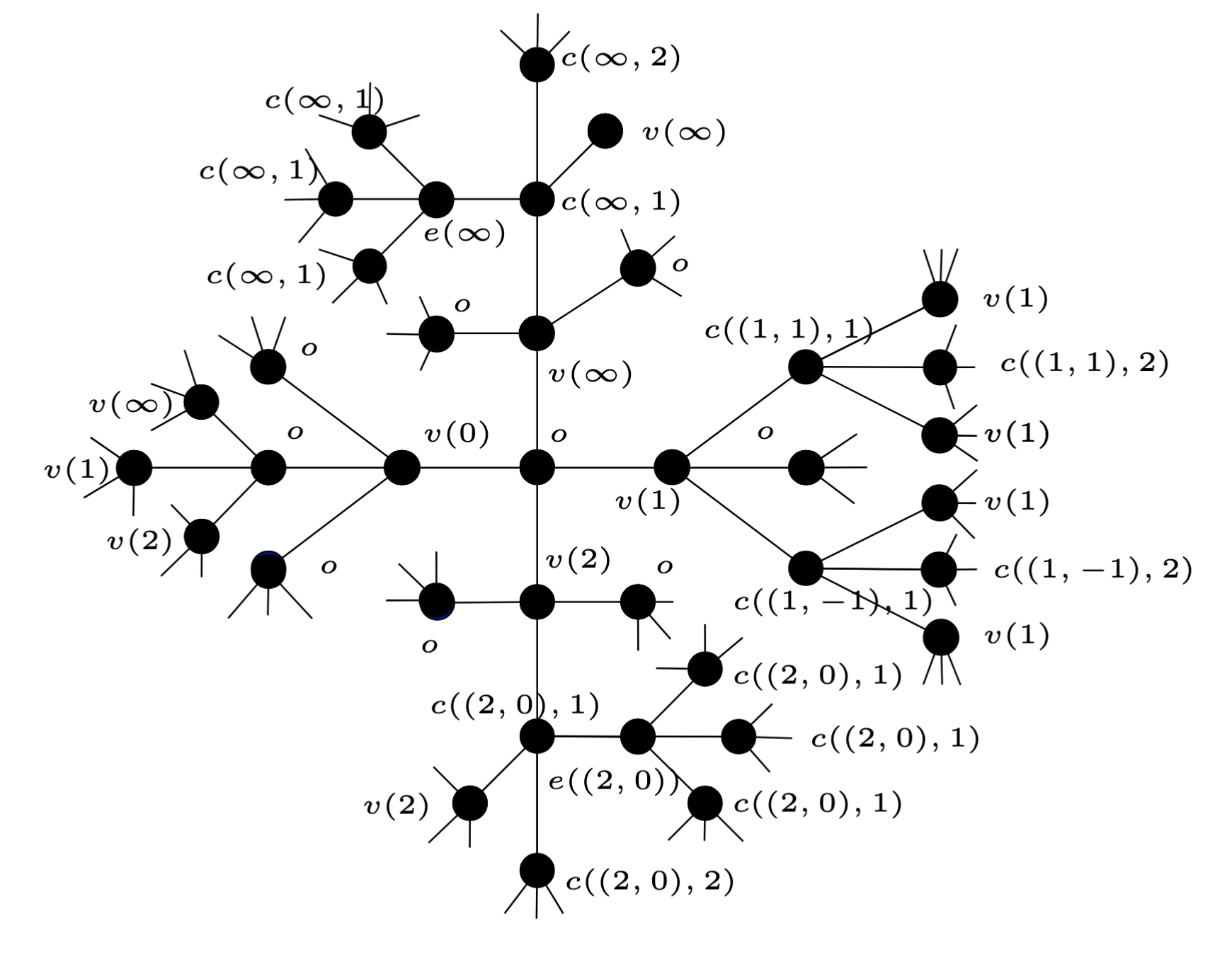}
\caption{An example of the labeling for finitely many vertices of $\mathcal{T}$ for $y^2=x^3+x-1$ over $\mathbb{F}_3$.}
\label{fig: tree all labels}
\end{figure}

An end corresponding to $v_1, v_2,\dots$ is rational if for some $M \in \mathbb{N}$, the sequence of vertices $(v_m)_{m \geq M}$ is bijective to the sequence of labels $c((x,y),n)_{n \geq N}$ for some $N \in \mathbb{N}$. We prove the following result.
\begin{lemma}
A rational end on $\mathcal{T}$ goes to a cusp $\alpha$ in $\mathbb{P}^1(F)$.
\end{lemma}
\begin{proof}
Lemma \ref{lem: end cusp s} covered the special case of rational ends on $\mathcal{S}$. A rational end $c((x,y),n)$ on $\mathcal{T}$ is $\overline\Gamma$-equivalent to a rational end on $\mathcal{S}$ which corresponds to an orbit-representative cusp. This means, for any cusp $\alpha \in k(E)$, there is a $\gamma \in \overline\Gamma$ such that $\alpha = \gamma \cdot \beta$ where the action is given by a fractional linear transformation and $\beta$ is an orbit-representative cusp. 
\end{proof}

\subsection{Defining Modular Symbols}
We define the group of our modular symbols using a similar definition to one given in \cite[Chapter 3]{stein}.
\begin{definition}\label{def: mod sym}
        The \textbf{group of modular symbols}, denoted $\mathbb{M}_2$, is the $\mathbb{Q}$-vector space over the set of symbols $[\alpha, \beta]$ where $\alpha, \beta \in \mathbb{P}^1(F)$ with the following two relations:
        \[
        [\alpha, \beta] = -[\beta, \alpha] \text{ and } [\alpha, \beta] + [\beta, \gamma] + [\gamma, \alpha] = 0.
        \]
        Let $\Gamma \subseteq \overline\Gamma$ be a congruence subgroup. The \textbf{group of modular symbols over $\Gamma$}, denoted $\mathbb{M}_2(\Gamma)$, is the quotient of $\mathbb{M}_2$ by the submodule generated by $[\alpha, \beta] - \gamma \cdot [\alpha, \beta]$ for $\gamma \in \Gamma$. The action $\gamma \cdot [\alpha, \beta]$ is given by $[\gamma \cdot \alpha, \gamma \cdot \beta]$ and the action on the cusps is by fractional linear transformations.
\end{definition}
On $\mathcal{T}$, a modular symbol $[\alpha, \beta]$ can be identified with a unique non-backtracking path from the cusp $\alpha$ to the cusp $\beta$. We say that the symbol $[\alpha, \beta]$ \textit{exits} the cusp $\alpha$ and \textit{enters} the cusp $\beta$. On $\mathcal{T}$, the sum of two modular symbols $[\alpha, \beta] + [\beta, \gamma]$ is the unique non-backtracking path $[\alpha, \gamma]$. 

\subsection{Reduced Symbols: A Combinatorial Approach}
In the next two sections, we define the generators. We first give a combinatorial construction and then an algebraic approach suitable for computations. 

The adjacent vertices of any $e(p)$-vertex are labeled $c(p,1)$. Choose one $c(p,1)$-vertex. The limit of the sequence $(c(p,n))_{n \geq 1}$, which is a rational end, as $n$ tends to infinity, is a cusp $\alpha$. For a different $c(p,1)$-vertex, the rational end $(c(p,n))_{n \geq 1}$ corresponds to a different cusp $\beta$. Consider the path identified by $[\alpha, \beta]$. The labels of the vertices of this path is the double-sided sequence:
\[\dots, c(p, 2), c(p, 1), e(p), c(p,1), c(p,2), \dots\]

For a vertex $v$ of type $s$, $o$, and $ns$, there are at least two distinct vertices with labels $c(p,1)$ and $c(p',1)$ of degree at most three from $v$. These two vertices give rational ends $(c(p,n))_{n \geq 1}$ and $(c(p',n))_{n \geq 1}$ corresponding to cusps $\alpha$ and $\beta$. In each case, the modular symbol $[\alpha, \beta]$ corresponds to a path on $\mathcal{T}$ with a specific double-sided infinite sequence of vertex labels given below. 
(Figures \ref{fig: uni_symbols} and \ref{fig: sym_on_tree} give visual representations).

\begin{itemize}
    \item The path given by the double-sided sequence with $p = \infty$ or $p = (l,m)$ where $x = l$ has a unique solution to $y = m$ to $f(x,y) = 0$. We call such symbols \textit{type $e$} since the cusps were found starting at the $e$-vertex.
    \begin{equation}
        \dots, c((x,y), 2), c((x,y), 1), e((x,y)), c((x,y),1), c((x,y),2),\dots \label{type e}
    \end{equation}
    \item The path given by the double-sided sequence with $p = (l,m)$ and $p' = (l,m')$ where $x = l$ gives two solutions of $y$ to $f(x,y) = 0$. These symbols we call \textit{type $s$} since the cusps were found starting at the $s$-vertex.
    \begin{equation}
        \dots, c(p,2), c(p,1), v(l), c(p',1), c(p',2),\dots \label{type s}
    \end{equation}
    \item The path given by the double-sided sequence with $p = (l,m)$ and $p' = (l',m')$ where $p, p'$ are any two rational points in $E(\mathbb{F}_q)$. These symbols we call \textit{type $o$} since the cusps were found starting at the $o$-vertex.
    \begin{equation}
        \dots, c(p, 2), c(p, 1), v(l), o, v(l'), c(p',1), c(p',2),\dots \label{type o}
    \end{equation}
    \item The path given by the double-sided sequence with $p = (l,m)$ and $p' = (l',m')$ where $p, p'$ are any two rational points in $E(\mathbb{F}_q)$ and $x = l'' \in \mathbb{F}_q$ yields no solution of $y$ to $f(x,y) = 0$. These symbols we call \textit{type $ns$} since the cusps were found starting at the $ns$-vertex.
    \begin{equation}
        \dots, c(p, 2), c(p, 1), v(l), o, v(l''), o, v(l'), c(p',1), c(p',2),\dots \label{type ns}
    \end{equation}  
\end{itemize}

\begin{figure}[htb]
\centering
\includegraphics[width=9cm, height = 5cm]{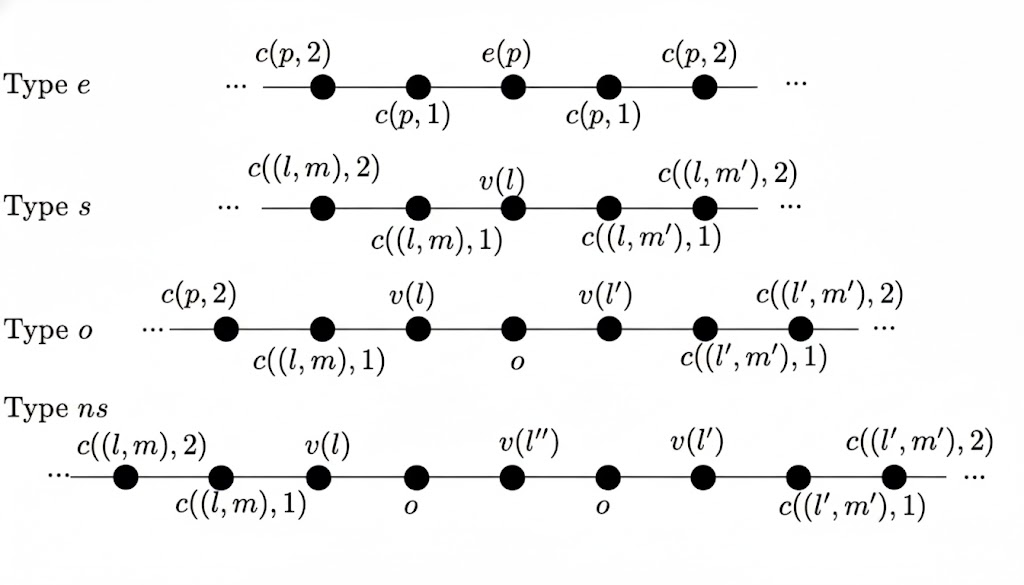}
\caption{The sequences of the four reduced symbols.}
\label{fig: uni_symbols}
\end{figure}

\begin{figure}[htb]
\centering
\includegraphics[width=8cm, height=7.5cm]{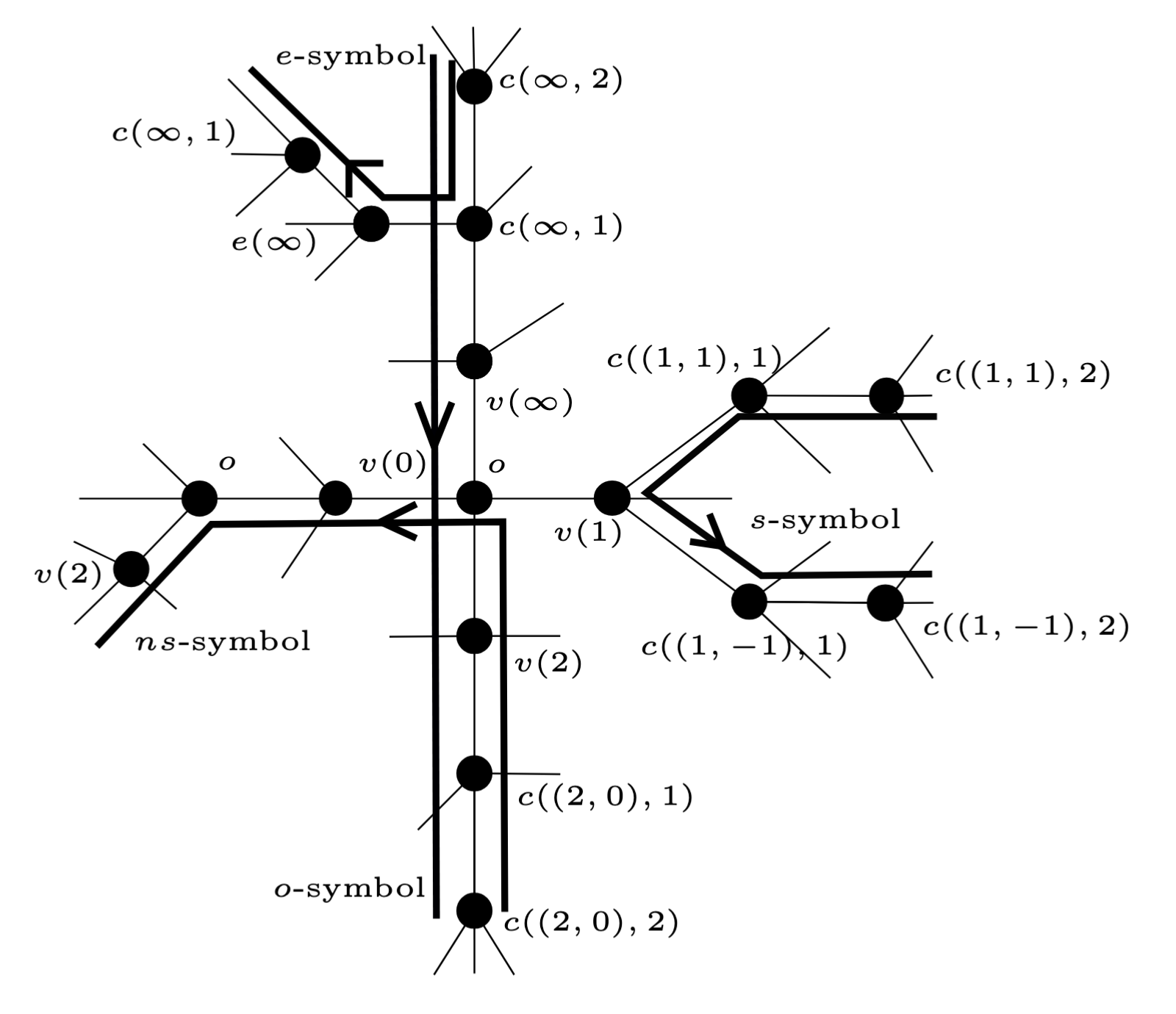}
\caption{An example of the four reduced symbols on $\mathcal{T}$ for $y^2=x^3+x-1$ over $\mathbb{F}_3$.}
\label{fig: sym_on_tree}
\end{figure}

\begin{definition}\label{reducedsymbol}
    The modular symbols that have the double-sided sequences of vertex labels given by (\ref{type e}), (\ref{type s}), (\ref{type o}), (\ref{type ns}) are called \textbf{reduced symbols}.
\end{definition}
\begin{definition}
The vertices of type $e,s,o$ and $ns$ are called \textbf{minimal vertices}, and we call the set of cusps whose pairs form a reduced symbol over a (unique) minimal vertex $v$ \textbf{the cusps attached to $v$}. 
\end{definition}
The minimal vertices are what distinguish the different types of reduced symbols. 
The double-sided sequences of the vertex labels corresponding to the four reduced symbols are unique. On $\mathcal{T}$, we can find infinitely many non-backtracking paths of these four types. However, we will show in the next section that these symbols are $\overline\Gamma$-translates of finitely many representatives, allowing for explicit computations.

\subsection{Reduced Symbols: An Algebraic Approach}
Each reduced symbol has a distinguishing minimal vertex. The stabilizer group of the minimal vertices on $\mathcal{S}$ can be found in \cite[Proposition 9]{takahashi}. Recall that we have the set of orbit-representative cusps on $\mathcal{S}$. Let $v$ be a minimal vertex with stabilizer group $\overline\Gamma_v \subset \overline\Gamma$ and let $\alpha$ be an orbit-representative cusp. Computing the action of elements of $\overline\Gamma_v$ on $\alpha$ will give us reduced symbols. 

Let $R$ be the set of orbit-representative cusps. The stabilizer groups of the minimal vertices can be explicitly computed using \cite[Proposition 6 and 9]{takahashi}. Furthermore, \cite[Theorem 5]{takahashi} tells us what finite groups the stabilizer groups are up to isomorphism.

Every elliptic curve has the rational point $\infty$. The stabilizer group of the vertex $e(\infty)$ on $\mathcal{S}$ is ${\rm{GL}_2}(\mathbb{F}_q)$. The group acts transitively on $\mathbb{P}^1(\mathbb{F}_q)$, which can be identified with the edges of $e(\infty)$ and hence the cusps attached to $e(\infty)$. The cusp $\infty$ is the only cusp in $R$ that can be used to find $e(\infty)$-symbols. The matrix $\begin{pmatrix}
0 & 1 \\
-1 & 0 
\end{pmatrix} \in {\rm{GL}_2}(\mathbb{F}_q)$ acts on $\infty$ to give us $0$. A reduced symbol of type $e(\infty)$ is the pair $[\infty, 0]$.

Let $p = (l,m)$ be a rational point on $E$ such that $x = l$ yields exactly one solution $y = m$ to $f(x,y) = 0$. The orbit-representative cusp corresponding to $p$ is $(y-m)/(x-l)$. The stabilizer group of the vertex $e(p)$ on $\mathcal{S}$ is isomorphic to ${\rm{GL}_2}(\mathbb{F}_q)$, which again acts transitively on the edges. The map is given by sending $x$ to $l$ and $y$ to $m$. Let $S_p \in \overline\Gamma_{e(p)}$ be the element that maps to 
$\begin{pmatrix}
0 & 1 \\
-1 & 0 
\end{pmatrix} \in {\rm{GL}_2}(\mathbb{F}_q)$. Applying $S_p$ to $(y-m)/(x-l)$, one reduced symbol of type $e(p)$ is the pair $[(y-m)/(x-l), S_p \cdot (y-m)/(x-l)]$.

Any ordered pair of cusps in $R$ form an $o$ or $s$-symbol. The stabilizer group of $o$ is the set of scalar matrices in ${\rm{GL}_2}(\mathbb{F}_q)$, which fixes every cusp. The stabilizer group of any $s$-vertex fixes the two cusps that form the $s$-symbol. This is because the cusps in $R$ are the $\overline{\Gamma}$-orbits of the cusps. An $o$- or $s$-symbol is any pair of cusps from $R$. Note that the $s$-symbols comes from the minimal vertices with label $v(l)$ where $x = l$ yields two solutions of $y$ in $ff(x,y) = 0$. 

Finally, \cite[Theorem 5]{takahashi} tells us that the stabilizer group of an $ns$-vertex $v(l')$ is isomorphic to $\mathbb{F}_q(\theta)^\times$, where $\mathbb{F}_q(\theta)$ is the quadratic extension of $\mathbb{F}_q$ containing $\theta \in \overline{\mathbb{F}_q}$ such that $f(l', \theta) = 0$. The group $\overline\Gamma_{v(l')}$ is cyclic. Let $P_{l'}$ be the generator and $R_{l'}^i = \{P_{l'}^i \cdot \alpha : \alpha \in R\}$ for $1 \leq i \leq q$. (We drop the subscript if it is clear which $ns$-vertex $v(l')$ we are referring to). On the $q+1$ edges of $v(l')$, the action of $P_{l'}$ is equivalent to the action of the $(q+1)$-cycle $(1 2 ... q+1)$ on the set $\{1, 2, ..., q+1\}$. This means that $P_{l'}^k$ sends $\alpha \in R_{l'}^i$ to the cusp that is $\overline\Gamma$-equivalent to $\alpha$ in $R_{l'}^{i'}$ where $i' < q+1$ such that $i' \equiv i+k \mod q+1$.

Using the action of the stabilizer groups on the minimal vertices and the orbit-representative cusps, we get the following result.
\begin{proposition}\label{thm: explicit red sym}
The reduced symbols over $\overline\Gamma$ can be represented as follows: 
\begin{itemize}
	\item The $e(\infty)$-symbols are $\overline\Gamma$-translates of $[\infty, 0]$,
	\item The $e(p)$-symbols are $\overline\Gamma$-translates of $[(y-m)/(x-l), S_p \cdot (y-m)/(x-l)]$,
	\item The $o$- and $s$-symbols are $\overline\Gamma$-translates of $[\alpha, \beta]$ where $\alpha, \beta \in R$,
	\item The $ns$-symbols are $\overline\Gamma$-translates of $[\alpha, \beta]$ where $\alpha \in R$, $\beta \in R_{l'}^i$, and $v(l')$ is an $ns$-vertex.
	\end{itemize}
\end{proposition}

\begin{remark}
The vertices of type $s$ or $ns$ have label $v(l)$ for $l \in \mathbb{F}_q$. If there is not a unique $s$- or $ns$-vertex on $\mathcal{S}$, it is better to distinguish the corresponding symbols by labeling them $v(l)$-symbols while keeping track of their type.
\end{remark}

\begin{remark}
    \normalfont The four reduced symbols serve as the analog to the ``basic" modular symbols over the rational function field $\mathbb{F}_q(T)$ \cite[pg 284]{teitelbaum}. Over $\mathbb{F}_q(T)$, there is one type of ``basic" modular symbols. The ``basic" symbols are the ${\rm{GL}_2}(\mathbb{F}_q[T])$-translates of $[\infty, 0]$, which correspond to our $e(\infty)$-symbols.
\end{remark}
\begin{example}\label{ex: example red sym}
	\normalfont We compute the reduced symbols over the elliptic curve $E$ given by equation $y^2=x^3+x-1$ over $\mathbb{F}_3$. The set of orbit-representative cusps are given by $R = \{\infty, y/(x+1), (y-1)/(x-1), (y+1)/(x-1)\}$. 

The isomorphism from $\overline\Gamma_{e((-1,0))}$ to ${\rm{GL}_2}(\mathbb{F}_3)$ is given by sending $x$ to $-1$ and $y$ to $0$. The cusp $y/(x+1)$ is the orbit-representative cusp used to find $e((-1,0))$-symbols. The other cusp we find using Proposition \ref{thm: explicit red sym} above is $(x^2-x+1)/y$. 

Note that any ordered pair of cusps in $R$ form an $o$ or $s$-symbol, corresponding to the vertex $v(1)$. The stabilizer group of $o$ is the set of scalar matrices in ${\rm{GL}_2}(\mathbb{F}_3)$, which fixes every cusp. The stabilizer group of $v(1)$ fixes the cusps $(y+1)/(x-1)$ and $(y-1)/(x-1)$. 

The stabilizer group of $v(0)$, the $ns$-vertex, is isomorphic to $\mathbb{F}_3(\sqrt{-1})^\times$. It has generator 
$P = \begin{pmatrix}
	y+1 & -x^2-1 \\
        x & -y+1
\end{pmatrix}$. Let $R^i$ be as defined in Proposition \ref{thm: explicit red sym}, which are shown below. The cusps are listed in the sets $R^i$ below by their $\overline\Gamma$-orbits. For example, the first cusps listed in each $R^i$ below: $\infty, (y+1)/x, y/x, (y-1)/x$ are $\overline\Gamma$-equivalent. 
\begin{gather*}
R = R^0 = \{\infty, y/(x+1), (y-1)/(x-1), (y+1)/(x-1)\} \\
R^1 = \left\{\dfrac{y+1}{x}, \dfrac{y-x^2+1}{-y+x+1}, \dfrac{x^2-1}{y-1}, \dfrac{-y+x^2+1}{y-x-1}\right\} \\
R^2 = \left\{\dfrac{y}{x}, \dfrac{x^2-1}{y}, \dfrac{x^2-y}{y-x}, \dfrac{x^2+y}{y+x}\right\} \\
R^3 = \left\{\dfrac{y-1}{x}, \dfrac{-y+x^2+1}{-y-x-1}, \dfrac{y+x^2+1}{y+x+1}, \dfrac{x^2-1}{y+1}\right\} 
\end{gather*}

The reduced symbols over $\overline\Gamma$ for our curve $E/\mathbb{F}_3$ are as follows:
\begin{itemize}
	\item The $e(\infty)$-symbols are $\overline\Gamma$-translates of $[\infty, 0]$,
	\item The $e((-1,0))$-symbols are $\overline\Gamma$-translates of $[y/(x+1), (x^2-x+1)/y]$,
	\item The $s$- or $v(1)$-symbols are $\overline\Gamma$-translates of $[(y-1)/(x-1), (y+1)/(x-1)]$,
	\item The $o$-symbols are $\overline\Gamma$-translates of $[\alpha, \beta]$ where $\alpha, \beta \in R$ except for the $s$-symbol given above,
	\item The $ns$ or $v(0)$-symbols are $\overline\Gamma$-translates of $[\alpha, \beta]$ where $\alpha \in R$ and $\beta \in R^i$ for $i = 1,2, 3$.
\end{itemize}
\end{example}

\subsection{Proving Theorem \ref{main 1}}
In order to prove Theorem \ref{main 1}, we must use a result from \cite{serre}. Serre shows that the vertices of $\mathcal{T}$ can be interpreted as vector bundles of rank 2 over $E$ \cite[pg 100]{serre}. Let $B$ be such a vector bundle and let $\text{det}(B)$ be the determinant of the bundle. The determinant is a line bundle and we denote the degree of $B$ (the degree of a rational section) as $\text{deg}(B)$. Let $C$ be a subbundle of rank 1. He defines the following invariant on the vertices:  
\[N(B; C) := 2\cdot \text{deg}(C) - \text{deg}(B)
\text{ and } 
N(B) := \sup_C N(B; C) \]

Let $N(E)$ denote the invariant on the vertices in our construction. Our vertices have the following invariants.
\begin{center}
\begin{tabular}{ |c|c| } 
 \hline 
 N(E) & Vertex type  \\ 
 \hline \hline
 $n \geq 1$ & $c((x,y),n)$  \\ 
 \hline
 $0$ & $os$, $s$, $e$  \\ 
 \hline
 $-1$ & $o$  \\
 \hline
 $-2$ & $ns$  \\
 \hline 
\end{tabular}
\end{center}

Recall that the vertices $e, s, o, \text{ and } ns$ are the minimal vertices. This name comes from this invariant. If a minimal vertex has invariant $n$, then at least two of the adjacent vertices have invariant $n+1$. This means that at a minimal vertex, at least two of the the adjacent vertices lead to cusps without needing to pass another minimal vertex. The other types of vertices do not have this property. A reduced symbol $[\alpha, \beta]$ is given by a double-sided sequence of vertices whose invariants strictly decrease then strictly increase. We use this to prove Theorem \ref{main 1}.

\begin{theorem}[Restatement of Theorem \ref{main 1}]\label{uni sum}
Any modular symbol $[\alpha, \beta]$ can be written as a finite sum of the reduced symbols.
\end{theorem}

\begin{figure}[htb]
\centering
\includegraphics[width=11cm, height=4cm]{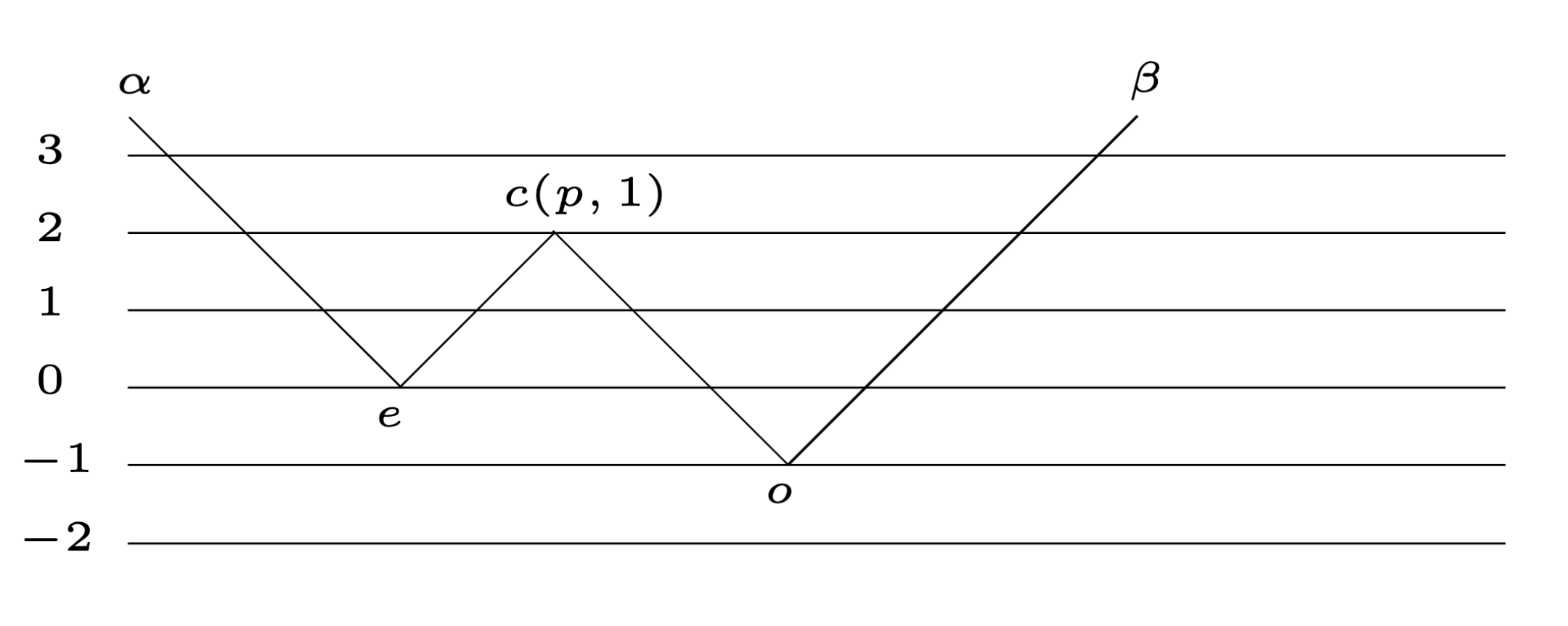}
\caption{A possible graph of the function $n(i)$ for a modular symbol $[\alpha, \beta]$.}
\label{fig: mod_graph}
\end{figure}

\begin{figure}[htb]
\centering
\includegraphics[width=11cm, height=4cm]{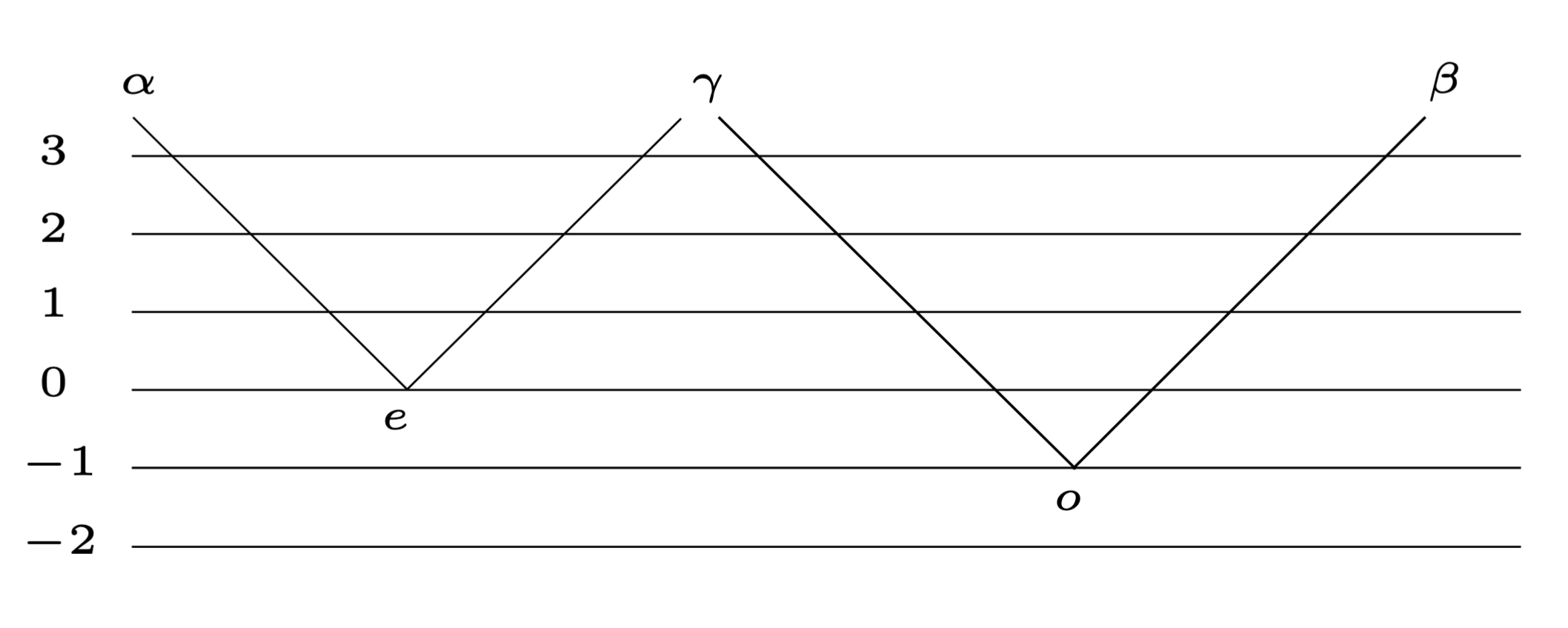}
\caption{The symbol $[\alpha, \beta]$ from Figure \ref{fig: mod_graph} as a sum of reduced symbols.}
\label{fig: unimod_graph}
\end{figure}

\begin{proof}
A modular symbol $[\alpha, \beta]$ on $\mathcal{T}$ can be identified with a double-sided sequence of vertex labels. Let $i \in \mathbb{Z}$ denote the $i$-th vertex and $n(i) \in \mathbb{Z}$ the Serre invariant of the vertex. We can graph the function $n(i)$ and extend linearly to $\mathbb{R}^2$. (See Figure \ref{fig: mod_graph} for an example). Since the symbol starts at cusp $\alpha$ and then goes to cusp $\beta$, $n(i)$ strictly decreases and then eventually strictly increases, meaning the symbol must pass through at least one minimal vertex. Recall that minimal vertices are the only vertices where there are at least two adjacent vertices of a higher invariant. The graph contains finitely many local minimums and maximums. The local minimums correspond to minimal vertices. The maximums correspond to any vertices other than those of invariant $-2$. 

The graph of the reduced symbols has exactly one (absolute) minimum and no local maximums. They are translates of the real-valued function $f(x) = |x|$. We decompose our original symbol over the vertices corresponding to the local minimums of the function $n(i)$. 
Suppose the modular symbol starts at a cusp $\alpha$, goes to the minimal vertex with coordinates $(i_1, n(i_1))$, up to the vertex corresponding to a local maximum $(j, n(j))$, down to the minimal vertex with coordinates $(i_2, n(i_2))$, and continues to $\beta$. That there is a path between $(i_1, n(i_1))$ and $(i_2, n(i_2))$ on $\mathcal{T}$ with no other minimal vertex between them means there is at least one cusp attached to both minimal vertices. If $(j, n(j))$ corresponds to a vertex of type $c(p,n)$, then there is a unique cusp attached to this vertex. If it is not a vertex of type $c(p,n)$, there is at least one cusp we can go to through this vertex. By construction, the reduced symbol corresponding to the minimal vertex at the minimum $(i_1, n(i_1))$ can enter or exit this cusp $\gamma$. Hence, we can decompose the original symbol at the minimal vertex corresponding to $(i_1, n(i_1))$ through the reduced symbol $[\alpha, \gamma]$. Now we have a sum of two symbols: $[\alpha, \gamma]$ and $[\gamma, \beta]$. Using induction on the number of local minimums gives us the desired result. (See Figure \ref{fig: unimod_graph} for an example).
\end{proof}

\section{Relations Among the Symbols and the Steinberg Homology}\label{sec: relations among symbols}
In this section, we find a method of computing the Steinberg homology, which allows us to find the relations between the reduced symbols.

\subsection{The Steinberg Module and Homology}
Let us define the Steinberg homology. The \textit{Tits building}, denoted $T_2$, associated with ${\rm{GL}_2}(F)$ is the simplicial complex where the vertices correspond to the cusps, elements in $\mathbb{P}^1(F)$, and there are no edges. (See \cite[Chapter 8]{kondo} for more details). An apartment in this building is the ordered pair $[\alpha, \beta]$, and we have the following result.
\begin{theorem}\textup{\cite[Theorem 11.3.1]{kondo}}
The Tits building $T_2$ has the homotopy type of a bouquet of 0-spheres.
\end{theorem}

Using the Tits building, we can define the Steinberg module and Steinberg homology as follows.
\begin{definition}\label{def: stein module}
The \textbf{Steinberg module}, denoted $St$, is the reduced $0$-homology group of $T_2$ over $\mathbb{Z}$. That is, $St = \tilde{H}_0(T_2, \mathbb{Z})$. 
\end{definition}

\begin{definition}
    The \textbf{Steinberg homology of $\Gamma$} is defined to be the homology group of $\Gamma$ with coefficients in the Steinberg module, denoted $H_*(\Gamma, St)$.
\end{definition}

\begin{definition}\label{def: steinberg hom q}
    The \textbf{Steinberg homology of $\Gamma$ with coefficients in $\mathbb{Q}$} is defined to be the homology group $H_*(\Gamma, St \otimes_{\mathbb{Z}} \mathbb{Q})$.
\end{definition}

It is this homology group that we are interested in since $\mathbb{M}_2 \cong St \otimes_{\mathbb{Z}} \mathbb{Q}$ and $\mathbb{M}_2(\Gamma) \cong H_0(\Gamma, St \otimes_{\mathbb{Z}} \mathbb{Q})$.

\subsection{The Sharbly Complex}
We use the results of \cite{brown} and \cite{ash_resolution} to find a method of computing the Steinberg homology. The first step is to construct resolutions of the Steinberg module. One such complex is the Sharbly complex.
\begin{definition}\textup{\cite[Definition 4]{ash_resolution}]}\label{def: sharbly}
The \textbf{Sharbly complex}, denoted $Sh_* = Sh_*(\mathcal{O})$, is the complex of $\mathbb{Z}[{\rm{GL}_2}(F)]$-modules defined as follows. As an abelian group, $Sh_k(\mathcal{O})$ is generated by symbols of $[v_1, ... , v_{k+2}]$, where the $v_i$ are nonzero vectors in $F^2$, modulo the submodule generated by the following relations:
\begin{itemize}
\item $[v_{\sigma{1}}, ... , v_{\sigma{k+2}}] - sgn({\sigma})[v_1, ... , v_{k+2}]$ for all permutations $\sigma$,
\item $[v_1, ..., v_{k+2}]$ if $v_1, ... , v_{k+2}$ do not span all of $F^2$,
\item $[v_1, ..., v_{k+2}] - [av_1, v_2, ... , v_{k+2}]$ for all $a \in F^\times$.
\end{itemize}
\end{definition}

The boundary operator $\partial: Sh_k \to Sh_{k-1}$ acts as follows:
\begin{equation}\label{eq: sharbly boundary map}
\partial([v_1, ... , v_{k+2}]) = \sum_{i=1}^{k+2} (-1)^{i-1} [v_1, ... , \hat{v_i}, ... , v_{k+2}]
\end{equation}
where $\hat{v_i}$ means to delete $v_i$. 

Writing the vector $v_i$ as $(v_{i1}, v_{i2})$, we can map $[v_1, v_2]$ to the modular symbol $[v_{11}/v_{12}, v_{21}/v_{22}]$. This is constant on the cosets of the group generated by the relations given in Definition \ref{def: sharbly} above, giving us a surjective ${\rm{GL}_2}(F)$-equivariant map $\phi_{Sh}: Sh_0 \to St$. This gives us the following result.
\begin{theorem}\textup{\cite[Theorem 5]{ash_resolution}}
The following is an exact sequence of ${\rm{GL}_2}(F)$-modules:
\begin{equation}\label{eq: sharbly res}
\dots \to Sh_k \to Sh_{k-1} \to \dots \to Sh_0 \xrightarrow{\phi_{Sh}} St \to 0.
\end{equation}
\end{theorem}

\begin{proof}
We can follow the argument of \cite{ash_resolution} making appropriate modifications. One first constructs the Lee and Szczarba resolution $\mathcal{C}_*(A)$ of the Steinberg module given in \cite[Chapter 2]{ash_resolution}. The resolution is constructed using $A$, a principal ideal domain. However, as stated at the beginning of \cite[Chapter 5]{ash_resolution}, replacing the ring with its field of fractions still gives a resolution and suffices for us. We can use $F$ to construct the chain complexes, $\mathcal{C}_*(F)$, giving us a resolution of $St$.

This chain complex gives another chain complex $\mathcal{C}'_*(F)$, constructed using the lines in $F^2$. The boundary maps of $\mathcal{C}'_*$ is induced from $\mathcal{C}_*$, and we get another resolution of $St$. The construction of this chain complex is discussed in \cite[Chapter 3]{ash_resolution}. 

Once we have constructed $\mathcal{C}'_*$, we can construct a map $\mathcal{C}'_k \to Sh_k$ as shown in \cite[page 6]{ash_resolution}. For each line $l_i \in \mathbb{P}^1(F)$, choose a primitive vector $v_i \in l_i \cap \mathcal{O}^2$. We can map $(l_1, ... l_{k+2})$ to $[v_1, ... , v_{k+2}]$, which extends to a ${\rm{GL}_2}(F)$-equivariant map $\mathcal{C}'_* \to Sh_*$. We can finish our proof by following the proof given for \cite[Theorem 5]{ash_resolution}, which shows the Sharbly complex forms a resolution of $St$. 
\end{proof}

A chain in $Sh_1$ is a modular symbol, and $Sh_n$ for $n \geq 2$ capture the relations between the symbols. There is a clear ${\rm{GL}_2}(F)$-action, and hence $\Gamma$-action on $Sh_*$, Using the resolution given in (\ref{eq: sharbly res}) and results of standard spectral sequences of double complexes given on \cite[page 169]{brown}, we have: 
\[E^1_{pq} = H_q(\Gamma, Sh_p \otimes_{\mathbb{Z}} \mathbb{Q}) \implies H_{p+q}(\Gamma, St \otimes_{\mathbb{Z}} \mathbb{Q})\] 

The homology groups $H_*(\Gamma, Sh_* \otimes_{\mathbb{Z}} \mathbb{Q})$ is the \textit{Sharbly homology of $\Gamma$ with coefficients in $\mathbb{Q}$}. It is defined in \cite[Definition 6]{ash_resolution} similarly to Definition \ref{def: steinberg hom q}. 
Since $F^{\times}$ is a finite set, \cite[Theorem 7]{ash_resolution} gives us the following isomorphism:
\begin{equation}\label{eq: sharbly isom}
H_*(\Gamma, Sh_* \otimes_{\mathbb{Z}} \mathbb{Q}) \cong H_*(\Gamma, St \otimes_{\mathbb{Z}} \mathbb{Q}).
\end{equation}
However, the isomorphism given in (\ref{eq: sharbly isom}) is computationally impractical. We need to use a ``smaller" complex contained in $Sh_*$ that allows for explicit computations.

\subsection{The Simplicial Complex $\mathcal{K}$}\label{sec: simplicial k}
The purpose of this section is to construct a simplicial complex $\mathcal{K}$ suitable for computations. This complex is an analog of the Voronoi complex, a polyhedral decomposition of the space of positive-definite quadratic forms built using Voronoi's explicit reduction theory (see \cite[Chapter 3]{philippe} and \cite[Chapter 5]{ash_resolution}). We want to show that the chain complex formed by $\mathcal{K}$ forms a resolution of the Steinberg module.

For each minimal vertex $v$, we construct a corresponding simplicial complex, which we call a \textit{$v$-complex} or a \textit{minimal vertex complex}. 
\begin{itemize}
\item For each $e(p)$-vertex, there are $q+1$ cusps whose pairs form $e(p)$ symbols. Letting the cusps correspond to the vertices, an $e(p)$-complex is the $(q+1)$-dimension complex formed by taking all $q+1$ vertices. 
\item For an $s$-vertex, there are two cusps that form $s$-symbols. An $s$-complex is an edge.
\item For an $o$-vertex, there are $q+1$ cusps whose pairs form either $o$- or $s$-symbols. Letting the cusps correspond to the vertices, on $o$-complex is the $(q+1)$-dimension complex formed by taking all $q+1$ vertices. This shows us that $s$-complexes are subcomplexes of an $o$-complex. We make sure to keep track of these edges.
\item Finally, for an $ns$-vertex, there are $(q+1)^2$ cusps whose pairs form either $o$-, $s$-, or $ns$-symbols. The $ns$-complex is the $(q+1)^2$-dimension complex formed by taking all the $(q+1)^2$ vertices. There are $q+1$ o-complexes contained inside of the $ns$-complex. We again make sure to keep track of these complexes. 
\end{itemize}

\begin{figure}[htb]
\centering
\includegraphics[width=5.5cm, height=5.5cm]{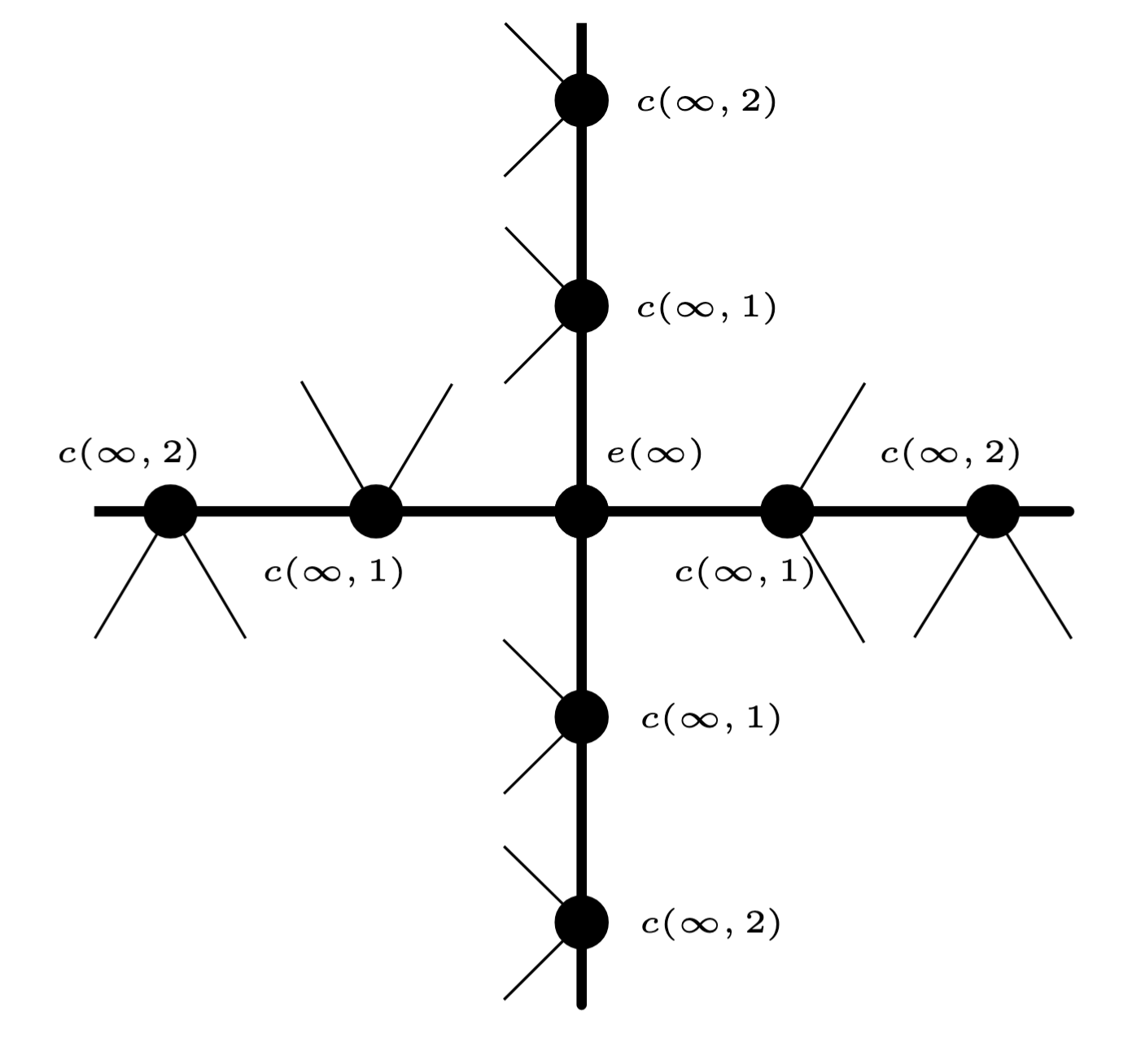}
\caption{The cusps attached to an $e(\infty)$-vertex that form a $e(\infty)$-complex for $y^2=x^3+x-1$ over $\mathbb{F}_3$.}
\label{fig: e_symbol}
\end{figure}

\begin{figure}[htb]
\centering
\includegraphics[width=4.5cm, height=5.5cm]{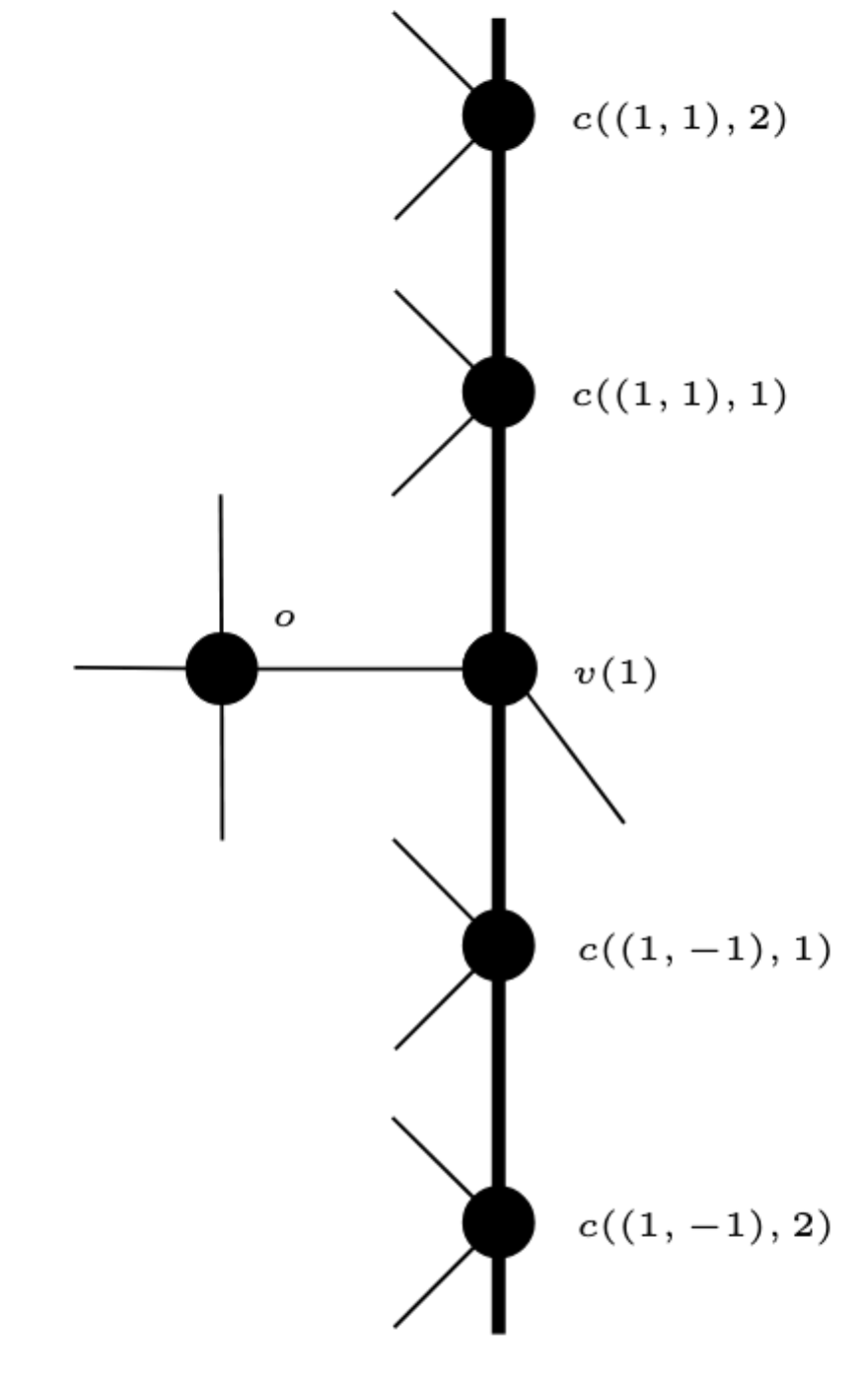}
\caption{The cusps attached to an $s$-vertex that form a $s$-complex for $y^2=x^3+x-1$ over $\mathbb{F}_3$.}
\label{fig: v2_symbol}
\end{figure}

\begin{figure}[htb]
\centering
\includegraphics[width=5.5cm, height=5.5cm]{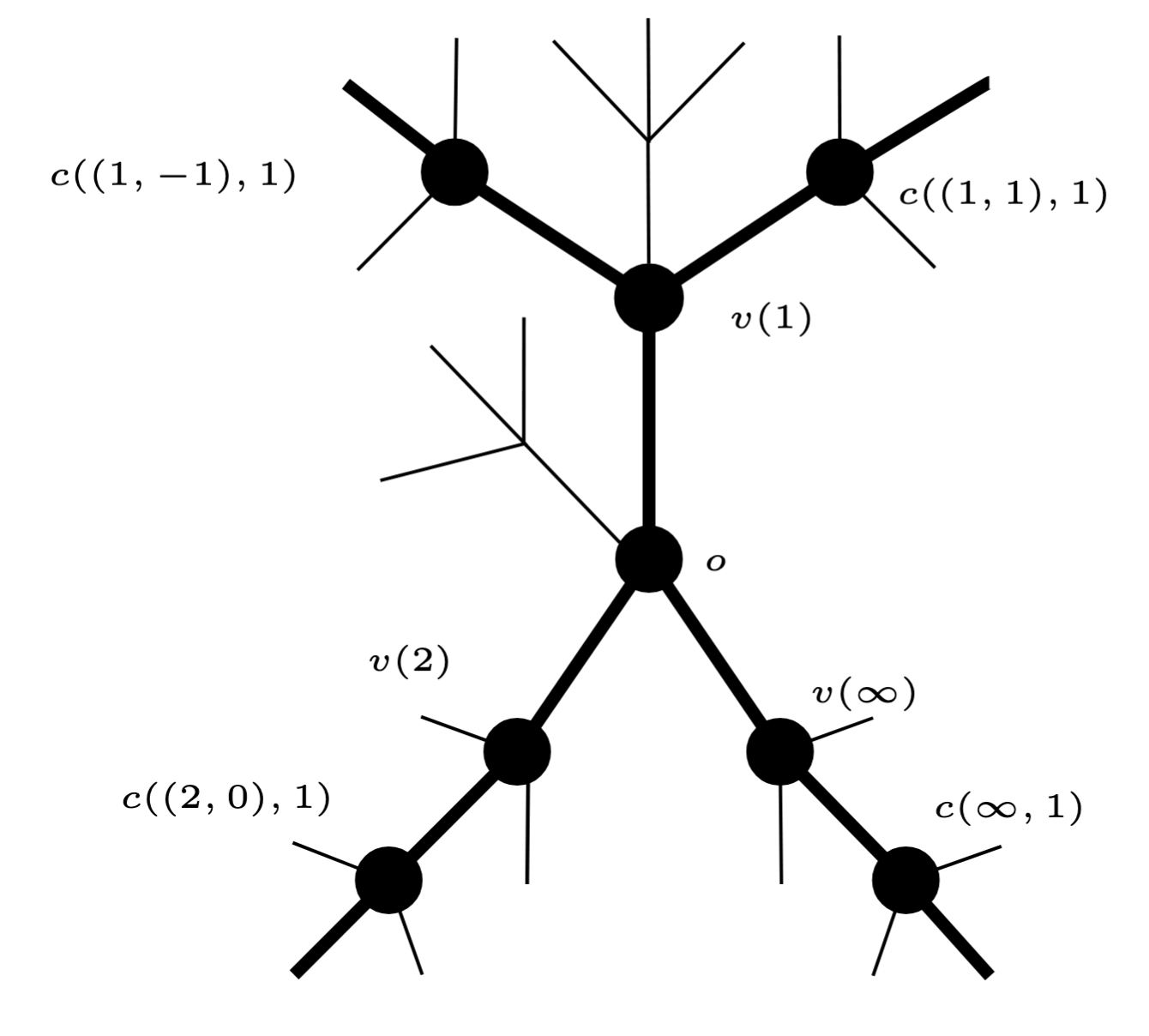}
\caption{The cusps attached to an $o$-vertex that form a $o$-complex for $y^2=x^3+x-1$ over $\mathbb{F}_3$.}
\label{fig:o_symbol}
\end{figure}

\begin{figure}[htb]
\centering
\includegraphics[width=6cm, height=6cm]{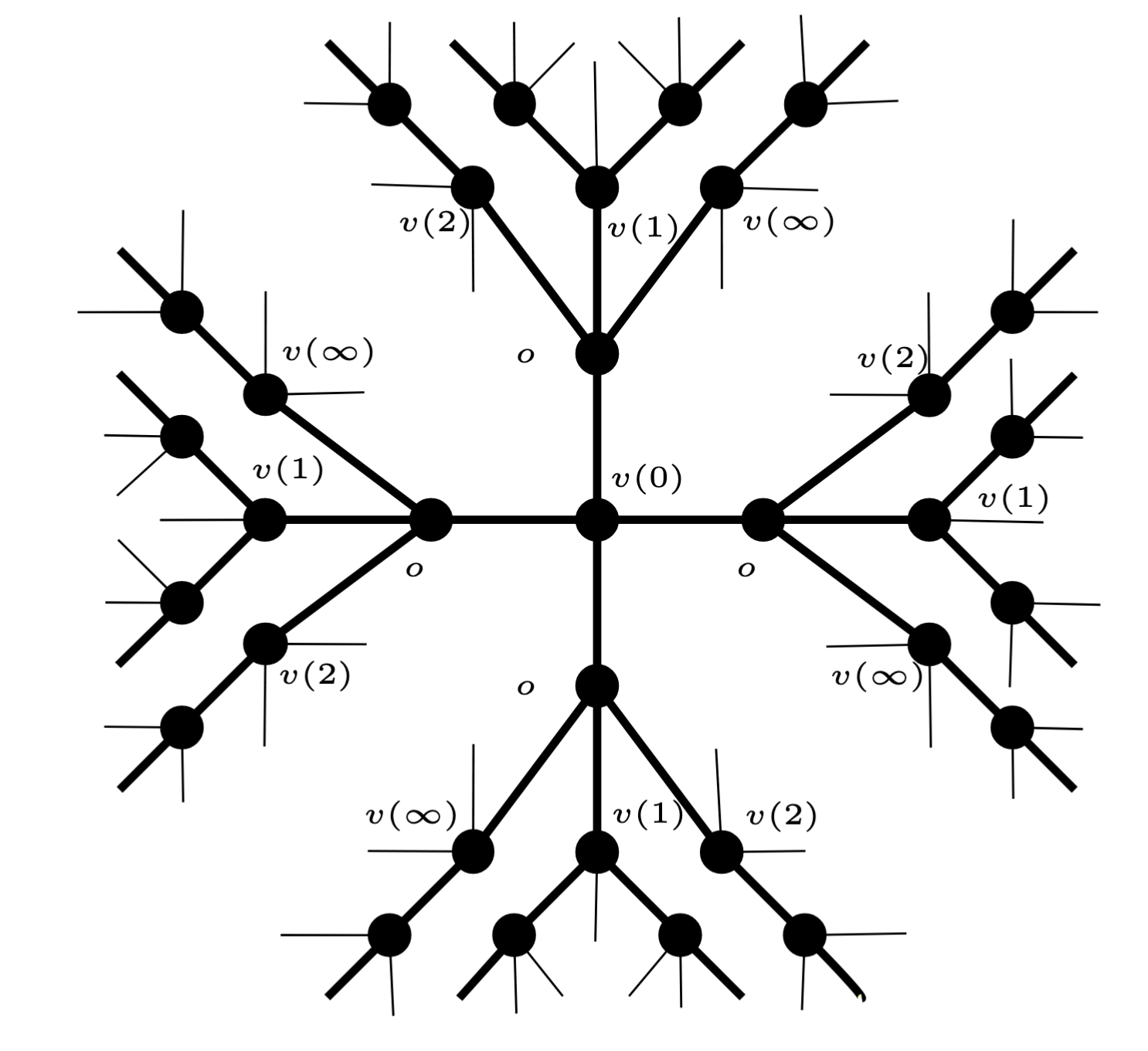}
\caption{The cusps attached to an $ns$-vertex that form a $ns$-complex for $y^2=x^3+x-1$ over $\mathbb{F}_3$.}
\label{fig: v_non_symbol}
\end{figure}

\begin{example}\label{ex: v cx}
\normalfont We return to Example \ref{ex: example red sym} with elliptic curve $E/\mathbb{F}_3$ given by the equation $y^2 = x^3 - x - 1$. The $e(\infty)$- and $e((-1,0))$-complexes are solid tetrahedrons. The $o$-complex is also a solid tetrahedron with one of its edges corresponding to the $s$-complex. Finally, an $ns$-simplex is a 16-dimension complex that contains 4 tetrahedrons corresponding to the $o$-complexes. Figures \ref{fig: e_symbol} to \ref{fig: v_non_symbol} show the cusps attached to a minimal vertex $v$ which form the $v$-complex.
\end{example}

We will build $\mathcal{K}$ using the individual $v$-complexes. First, we define what it means for two minimal vertices $v$ and $w$ to \textit{interact}. Let $C_v, C_w$ be the set of cusps on $v$ and $w$ respectively whose pairs formed reduced symbols i.e. the set of cusps attached to $v$ and $w$. If the intersection is empty, $v$ and $w$ have \textit{no interactions}. If the intersection contains one element, we say $v$ and $w$ \textit{interact trivially} and there is more than one element, then they \textit{interact non-trivially}. 

\begin{figure}[htb]
\centering
\includegraphics[width=9cm, height=11cm]{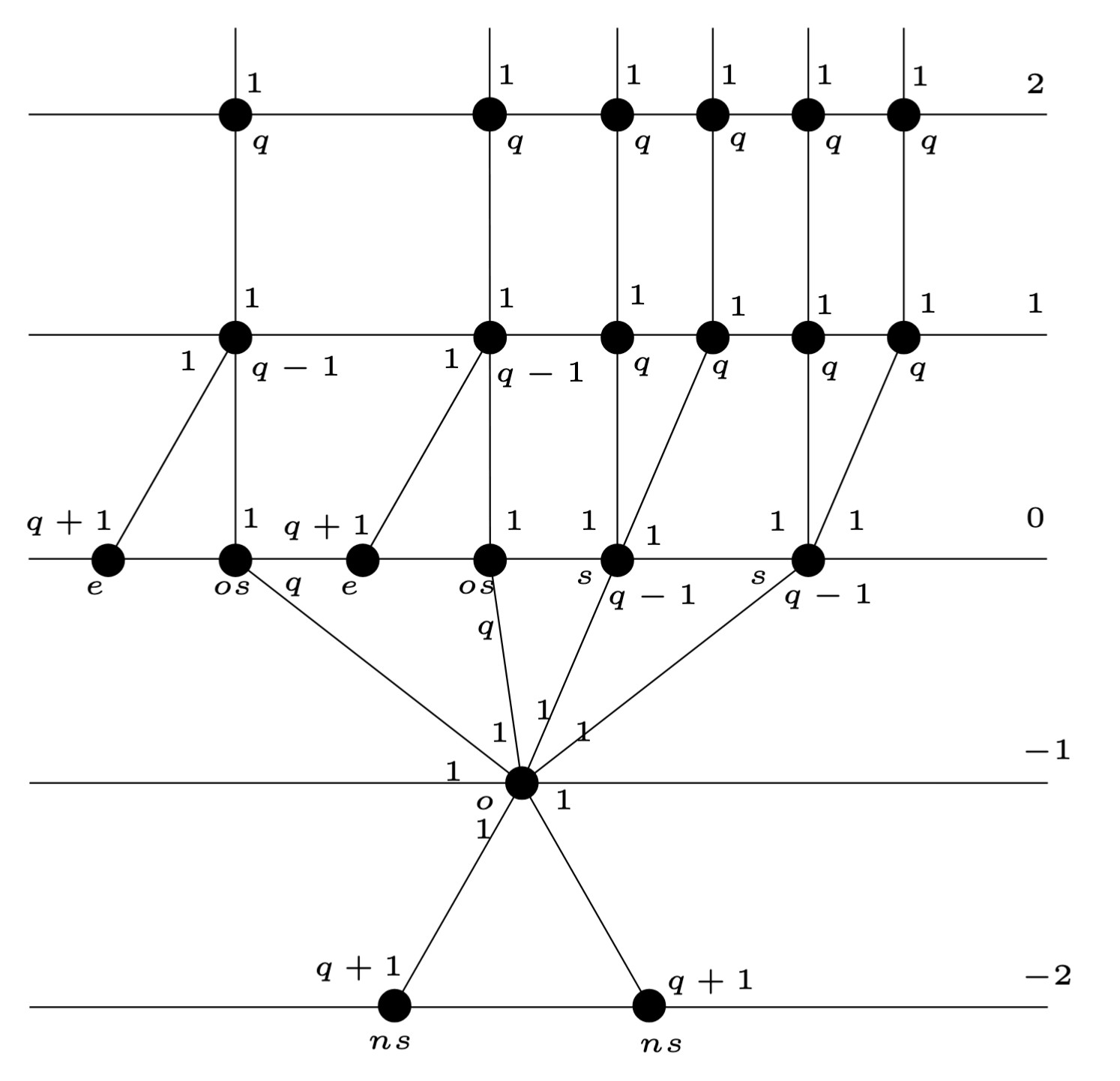}
\caption{The subtree $\mathcal{S}$ where each vertex is matched to its invariant and has the number of adjacent vertices one would see on $\mathcal{T}$. For $N(E) \geq 1$, the vertices have label $c(p,N(E))$ }\label{fig: quot_to_large}
\end{figure}

For an elliptic curve $E$, we know how to construct $\mathcal{T}$ and label each of its vertices. This allows us to map each vertex $v$ of $\mathcal{S}$ to its Serre invariant and identify how many adjacent vertices of $v$ are of a certain type. (See Figure \ref{fig: quot_to_large}). These identifications on $\mathcal{S}$ and observation of the labels of the vertices on $\mathcal{T}$ tells us how different minimal vertices interact.
We use this to define the notion of \textit{interactions} between the minimal vertex complexes. This will show us that we can attach a $v$-complex and $w$-complex by identifying their shared subcomplex in $\mathcal{K}$.

Given minimal vertices $v$ and $w$, the following $v$-and $w$-complexes do not share a common subcomplex and hence cannot be attached to each other in $\mathcal{K}$. We say that there is \textit{no interaction between the $v$- and $w$-complexes}. Such interactions are given by the following.

\begin{itemize}
\item An $s$-complex and an $e$-complex,
\item Two $e$-complexes with labels $e(p)$ and $e(p')$ where $p \neq p'$,
\item Two $s$-complexes with labels $v(l)$ and $v(l')$ where $l \neq l'$,
\item Two $ns$-complexes with labels $v(l)$.
\end{itemize}

Next, the following $v$- and $w$-complexes could share a common vertex (a cusp). Two such complexes are attached to each other by identifying their shared vertex in $\mathcal{K}$. We say that there is a \textit{trivial interaction between the $v$- and $w$-complexes}. Such interactions are given by the following.
\begin{itemize}
\item An $e$-complex both with labels $e(p)$,
\item Two $s$-complexes both with labels $v(l)$, 
\item Two $o$-complexes,
\item Two $ns$-complexes with labels $v(l)$ and $v(l')$ where $l \neq l'$. 
\end{itemize}

Finally, the following $v$- and $w$-complexes could share a subcomplex of dimension greater than one. Such two complexes are attached to each other by identifying their shared subcomplex in $\mathcal{K}$. We say that there is a \textit{non-trivial interaction between the $v$- and $w$-complexes}. Such interactions are given by the following.
\begin{itemize}
\item Two $o$-complexes could share an $s$-complex with label $v(l)$,
\item Two $ns$-complexes with labels $v(l)$ and $v(l')$ with $x \neq x'$ could share an $o$-complex. 
\end{itemize}
For trivial and non-trivial interactions, we use the phrase ``could share" because it is possible that on $\mathcal{T}$, two minimal vertices $v$ and $w$ do meet the criteria of labels but are two far apart that $C_v \cap C_w = \emptyset$.

We now construct $\mathcal{T}$ and $\mathcal{K}$ simultaneously from scratch. To avoid confusion, we use the terms ``vertices" and ``edges" to describe $\mathcal{T}$ and ``$v$-complexes" or ``minimal vertex complex" to describe $\mathcal{K}$. First, we start off with one $ns$-vertex $v_0$. This gives us a $ns$-complex, which we label $\mathcal{K}_0$. Then we add the adjacent vertices to $v_0$. These all correspond to vertices labeled $o$. On $\mathcal{K}_0$, these $o$-complexes are already present inside the $ns$-complex. We call this complex $\mathcal{K}_1$ and $\mathcal{K}_0 = \mathcal{K}_1$. We then add the next set of adjacent vertices on $\mathcal{T}$. In this case, there may be $s$-vertices, which are already present in the $o$-complexes. But there may be additional $ns$-vertices corresponding to a different type of $ns$-vertex from $v_0$. These $ns$-complexes are added to $\mathcal{K}_1$ by identifying their shared $o$-complex. This is now $\mathcal{K}_2$. We continue inductively in this way. The way that we attach new $v$-complexes appearing from new vertices showing up on $\mathcal{T}$ is by identifying the shared minimal vertex complex as 
described by the interactions between the complexes. For each $n$, we have that $\mathcal{K}_{n-1} \subseteq \mathcal{K}_n$. Taking the direct limit gives us a simplicial complex, which we call $\mathcal{K}$. This is a $(q+1)^2$-dimension complex. It is also a CW-complex equipped with the CW-topology.  We need the following result to show that we can use $\mathcal{K}$ to obtain a resolution of $St$.

\begin{lemma}\label{lem: k contract}
The simplicial complex $\mathcal{K}$ is contractible.
\end{lemma}
\begin{proof}
We first need to prove and use Lemmas \ref{lem: fund group kn trivial} and \ref{lem: kn contract}.
\begin{lemma}\label{lem: fund group kn trivial}
For all $n$, $\pi_1(\mathcal{K}_n)$ is trivial. 
\end{lemma}

\begin{proof}
Suppose that this is not the case. First, suppose a $v$-complex gets attached to $\mathcal{K}_{n-1}$ at two different vertices of $\mathcal{K}_{n-1}$ labeled $\alpha$ and $\beta$ where $[\alpha, \beta]$ does not form an edge. It suffices to use vertices since every subcomplex that is a minimal vertex complex contains at least one vertex. This implies we now have a reduced symbol $[\alpha, \beta]$ of type $v$. However, the fact that $\alpha, \beta$ are two different vertices of $\mathcal{K}_{n-1}$ tells us that there is already a unique non-backtracking path between the cusp $\alpha$ and $\beta$ on $\mathcal{T}$ that is not a reduced symbol. So this path goes through at least two minimal vertices on $\mathcal{T}$. In particular, the path cannot correspond to a $v$-symbol as we are claiming. This means there is a loop on $\mathcal{T}$, but this is impossible because $\mathcal{T}$ is contractible by \cite[Theorem 2.1]{grayson}.
 
Second, suppose two $v$-complexes $V_1$ and $V_2$ get attached to $\mathcal{K}_{n-1}$ at vertices $\alpha_1$ and $\alpha_2$ respectively. Suppose further that $V_1$ and $V_2$ are attached to each other at vertex $\beta$. We already have some path from vertices  $\alpha_1$ to $\alpha_2$ on $\mathcal{K}_{n-1}$, meaning there is a unique non-backtracking path $P_1$ from cusps $\alpha_1$ to $\alpha_2$ on $\mathcal{T}$. However, now there is a new path $P_2$ on $\mathcal{K}_n$ from vertices $\alpha_1$ to $\alpha_2$ through vertex $\beta$. On $\mathcal{T}$, this implies that there is a path from cusps $\alpha_1$ to $\alpha_2$ that passes through the minimal vertices $v_1$ and $v_2$ corresponding to $V_1$ and $V_2$ respectively. On $P_2$, the vertex $\beta$ corresponds to a rational end $c(p,n)$. No vertex of this end can be present in the path on $\mathcal{T}$ given by $P_1$. So $P_1$ and $P_2$ do not give the same path on $\mathcal{T}$. This implies that there is a loop on $\mathcal{T}$, which is impossible.
\end{proof}

\begin{lemma}\label{lem: kn contract}
For all $n$, $\mathcal{K}_n$ is contractible.
\end{lemma}

\begin{proof}
 Clearly, each $v$-complex is contractible. In particular, $\mathcal{K}_0$, an $ns$-complex, is contractible to some point $p$ on $\mathcal{K}_0$. Let $S_n = \mathcal{K}_n \setminus \mathcal{K}_{n-1}$. By Lemma \ref{lem: fund group kn trivial}, the complex $S_n$ is a set of disjoint $v$-complexes, each of which can be retracted to the subcomplex it shares with $\mathcal{K}_{n-1}$. This retraction can be made explicit. Let $v_1, ... , v_m$ denote the basis vectors of a $v$-complex $V$ in $S_n$ imbedded in $\mathbb{R}^m$. The $v_i$ also correspond to the vertices of $V$. The sum $\sum_{i=1}^m
c_i v_i$ such that $\sum_{i=1}^m c_i = 1$ is the coordinate of a point on $V$. Suppose we want to retract $V$ onto the complex given by vertices $v_1, ... , v_m'$ where $m' < m$. Let $A(t) = (a_1t + a_2t + \dots + a_{m'}t + a_{m'+1} \dots + a_m)/m'$ and $B(t) = (1-t)(a_1 + \dots a_{m'})/(m-m')$. Then the retraction is given by $R(t) = \sum_{i=1}^{m'} A(t) v_i + \sum_{j = m'+1}^m B(t) v_j$ for $t \in [0,1]$. At $t=0$, we have the complex $V$, and at $t=1$, we have the complex given by the vertices $v_1, \dots , v_{m'}$. By induction, each $\mathcal{K}_n$ is contractible to $p$. 
\end{proof}
\[ 
\begin{tikzcd}
\mathcal{K}_0 \arrow[hookrightarrow]{r} \arrow[swap]{d}{f_0} & \mathcal{K}_1 \arrow[hookrightarrow]{r} \arrow[swap]{d}{f_1} & \dots \arrow[hookrightarrow]{r} &\mathcal{K}_n \arrow[hookrightarrow]{r} \arrow[swap]{d}{f_n} & \dots\\%
p \arrow{r}{id} & p \arrow{r}{id} & \dots \arrow{r}{id} & p \arrow{r}{id} & \dots
\end{tikzcd}
\]
The diagram above clearly commutes. Each $\mathcal{K}_n$ is homotopy equivalent to the point $p$ where $f_n$ is the contraction map. The complex $\mathcal{K}$ is the direct limit of the $\mathcal{K}_n$ through inclusion and $p$ is the direct limit of the bottom sequence. Then \cite[Lemma 2.1.10]{ponto} tells us that $\mathcal{K}$ is homotopy equivalent to $p$, proving Lemma \ref{lem: k contract}.
\end{proof}

Let $Ch_n(\mathcal{K})$ be the chain group formed by the $n$-cells of $\mathcal{K}$. The vertices of $\mathcal{K}$ is the set of cusps, which we denote as $\{\text{cusps}\}$. Let $\mathcal{C}_*(\mathcal{K})$, or $\mathcal{C}_*$, be the relative cellular chain complex of $\mathcal{K}$ with respect to the vertices of $\mathcal{K}$. That is, $\mathcal{C}_n(\mathcal{K}) = Ch_n(\mathcal{K})/Ch_n(\{\text{cusps}\})$. From this, we see that $\mathcal{C}_n(\mathcal{K}) = Ch_n(\mathcal{K})$ for $n \geq 1$. The boundary map of $\mathcal{C}_*$ is induced by $Ch_*$, which in turn is given by the boundary map of $Sh_*$ given in (\ref{eq: sharbly boundary map}). We now have a resolution of $St$ using $\mathcal{K}$. 
\begin{theorem}\label{thm: k resolution}
The following is an exact sequence of ${\rm{GL}_2}(F)$-modules:
\begin{equation}\label{eq: k res}
0  \to \mathcal{C}_{(q+1)^2} \dots \to \mathcal{C}_k \to \mathcal{C}_{k-1} \to \dots \to \mathcal{C}_2 \xrightarrow{\partial_2} \mathcal{C}_1 \xrightarrow{\phi} St \to 0.
\end{equation}
\end{theorem}
\begin{proof}
The map $\phi$ takes a 1-cell $(\alpha, \beta)$ and maps it to its class, or the modular symbol $[\alpha, \beta]$, in $St$. This map is ${\rm{GL}_2}(F)$-equivariant and surjective by Theorem \ref{uni sum}. Next, we need to show that kernel of $\phi$ is equal to the image of $\partial_2$. This is rather long, and so we dedicate Section \ref{sec: res stein} to its proof. The exactness at all the other places follows by Lemma \ref{lem: k contract}.
\end{proof}

The complex is contractible, in particular, acyclic. There is a clear ${\rm{GL}_2}(F)$-action, and hence $\Gamma$-action on $\mathcal{K}$. The stabilizer groups of the cells in $\mathcal{C}_*$ are finite as we will later prove. Like the Sharbly complex, we get the following result on the spectral sequence of $\mathcal{C}_*$:
\[E^1_{pq} = H_q(\Gamma, \mathcal{C}_p \otimes_{\mathbb{Z}} \mathbb{Q}) \implies H_{p+q}(\Gamma, St \otimes_{\mathbb{Z}} \mathbb{Q}) \]
The homology groups $H_*(\Gamma, \mathcal{C}_* \otimes_{\mathbb{Z}} \mathbb{Q})$ are called the \textit{equivariant relatively homology groups of $(\Gamma, \mathcal{K})$}. Tensoring our coefficient module with $\mathbb{Q}$, a field of characteristic zero, ensures that $E^1_{pq} = 0$ for all $q \neq 0$ and that the spectral sequence collapses at $E^2$. 

Using a similar argument to the proof of \cite[Theorem 7]{ash_resolution}, we get the following result, the analog to \cite[Corollary 12]{ash_resolution}.
\begin{corollary}\label{cor: k steinberg isom}
There is an isomorphism:
\begin{equation}
H_*(\Gamma, \mathcal{C}_* \otimes_{\mathbb{Z}} \mathbb{Q}) \cong H_*(\Gamma, St \otimes_{\mathbb{Z}} \mathbb{Q})
\end{equation}
\end{corollary}
Since $\mathcal{C}_*$ encapsulates information on the reduced symbols, which are finite modulo the action of $\Gamma$, and their relations, the isomorphism above gives us a better chance at computing the Steinberg homology. 

\subsection{The Differential Map}
In order to compute the Steinberg homology, we need to find the relations of the reduced symbols. To do this, we define a differential map that coincides up to sign with the differential map on $E^1_{p0}$, using the results of equivariant relative homology and spectral sequences given in \cite{brown} and \cite{philippe}. 

First, induce a global orientation on all cells of $\mathcal{C}_*$. Let $\Sigma^*_p = \Sigma_p^*(\Gamma)$, for $p \geq 1$, denote the set of representatives of cells in $\mathcal{C}_p$ modulo the action of $\Gamma$. Let $\Sigma_p \subseteq \Sigma_p^*$ denote the set of cells $\sigma$ such that any element of the stabilizer group $\Gamma_{\sigma}$ preserves the orientation on $\sigma$. Let $V_p$ be the free abelian group generated by $\Sigma_p$.

We define the following map:
\begin{equation}
d_p: V_p \to V_{p-1}.
\end{equation}
Let $\sigma \in \Sigma_p$ be a $p$-cell with orientation given by $(v_1, v_2, \dots , v_p)$ where $v_j$ are the vertices. Let $\tau' \in \Sigma_{p-1}$ be a face of $\sigma$ with orientation $(v_1, v_2, \dots , \hat{v_i}, \dots v_p)$. Then we define $\epsilon(\tau', \sigma) = (-1)^{i+1}$. Now let $\tau \in \Sigma_{p-1}$ be the unique $\Gamma$-representative of $\tau'$ and $\gamma \in \Gamma$ such that $\tau = \gamma \cdot \tau'$. We set $\eta(\tau, \tau') = 1$ if the orientations are compatible and $-1$ if not. 
Using this, we get
\begin{equation}\label{eq: our_diff}
    d_p(\sigma) = \displaystyle\sum_{\tau \in \Sigma_{p-1}} \displaystyle\sum_{\tau'} \eta(\tau, \tau')\epsilon(\tau',\sigma)\tau.
\end{equation}

According to \cite[Equation 7.7]{brown}, we have the following:
\begin{equation}\label{eq: spec1}
E^1_{pq} = \bigoplus_{\sigma \in \Sigma_p^*} H_q(\Gamma_{\sigma}, \mathbb{Q}_{\sigma}) \implies H_{p+q}(\Gamma, \mathcal{C}_* \otimes_{\mathbb{Z}} \mathbb{Q})
\end{equation}
The coefficient module $\mathbb{Q}_{\sigma}$ is defined as follows. Let $\mathbb{Q}_{\sigma}$ be the $\Gamma_{\sigma}$-module that is additively isomorphic to $\mathbb{Q}$ and acts by $1$ if $\gamma \in \Gamma_{\sigma}$ preserves the orientation of $\sigma$ and $-1$ if not.

As stated earlier, since $\mathbb{Q}$ is a field of characteristic 0, the terms in (\ref{eq: spec1}) vanish if $q > 0$, and the spectral sequence collapses at $E_2$. If $\Gamma_{\sigma}$ contains an element that changes the orientation of $\sigma$, then 2 kills $H_0(\Gamma_{\sigma},\mathbb{Q}_{\sigma})$. If not, then $H_0(\Gamma_{\sigma},\mathbb{Q}_{\sigma}) \cong \mathbb{Q}_{\sigma}$. This means modulo $\mathcal{S}_2$, the Serre class of order a power of 2, Equation (\ref{eq: spec1}) becomes:
\[E^1_{p0} = \bigoplus_{\sigma \in \Sigma_p} \mathbb{Q}_{\sigma}.\]
There is an isomorphism between $E^1_{p0}$ and $V_p$ given by the choice of orientation for each $\sigma \in \Sigma_p$.

In \cite[Chapter VII, Proposition 8.1]{brown}, there is a differential map:
\[d^1_p : E^1_{p0} \to E^1_{(p-1)0},\]
and we show that our map $d_p$ is this map up to a sign. 

Given $\sigma \in \Sigma^*_p$ with a face $\tau'$, let $\Gamma_{\sigma\tau} = \Gamma_{\sigma} \cap \Gamma_{\tau}$. We have the transfer map: 

\[t_{\sigma\tau'} : H_*(\Gamma_{\sigma}, \mathbb{Q}_{\sigma}) \to H_*(\Gamma_{\sigma\tau'}, \mathbb{Q}_{\sigma}).\]

There is a natural map $\mathbb{Q}_{\sigma} \to \mathbb{Q}_{\tau'}$ that induces a map on the homology:
\[u_{\sigma\tau'} : H_*(\Gamma_{\sigma\tau'}, \mathbb{Q}_{\sigma}) \to H_*(\Gamma_{\tau'}, \mathbb{Q}_{\tau'}).\]

Finally, suppose $\tau \in \Sigma_{n-1}^*$ is the representative of $\tau'$ modulo the action of $\Gamma$. So there exists $\gamma \in \Gamma$ such that $\tau' = \gamma \cdot \tau$ and $\gamma$ induces the followinng isomorphism:

\[v_{\tau'\tau}: H_*(\Gamma_{\tau'}, \mathbb{Q}_{\tau'}) \to H_*(\Gamma_{\tau}, \mathbb{Q}_{\tau}).\]

The restriction of $d^1_p$ to $H_*(\Gamma_{\sigma}, \mathbb{Q}_{\sigma})$ is equal, up to sign, to 
\begin{equation}\label{eq:brown_diff}
    \displaystyle\sum_{\tau'} v_{\tau'\tau} u_{\sigma\tau'} t_{\sigma\tau'}
\end{equation}
where $\tau'$ represents the faces of $\sigma$ modulo the action of $\Gamma_{\sigma}$. 

It remains to show that our map $d_p$ coincides with the map $d^1_p$ given by \cite{brown} up to a sign. 

Suppose that $\tau \in \Sigma_{p-1}$, then the orientation module $\mathbb{Q}_{\sigma}$ is equal to $\mathbb{Q}$, and so we have that
$$v_{\tau'\tau}: H_0(\Gamma_{\tau'}, \mathbb{Q}_{\tau'}) = \mathbb{Q} \to H_0(\Gamma_{\tau}, \mathbb{Q}_{\tau}) = \mathbb{Q}$$
is multiplication by $\eta(\tau, \tau')$. 

When $\sigma \in \Sigma_p$, then again $\mathbb{Q}_{\sigma} = \mathbb{Q}$, and we have that 
$$u_{\sigma\tau'}: H_0(\Gamma_{\sigma\tau'}, \mathbb{Q}_{\sigma}) = \mathbb{Q} \to H_0(\Gamma_{\tau'}, \mathbb{Q}_{\tau'}) = \mathbb{Q}$$
is multiplication by $\epsilon(\tau',\sigma)$, up to a sign depending on $p$.

Finally, the map
$$t_{\sigma\tau'}: H_0(\Gamma_{\sigma}, \mathbb{Q}_{\sigma}) = \mathbb{Q} \to H_0(\Gamma_{\sigma\tau'}, \mathbb{Q}_{\sigma}) = \mathbb{Q}$$
is multiplication by $k = [\Gamma_{\sigma} : \Gamma_{\sigma\tau'}]$. Multiplying the sum (\ref{eq:brown_diff}) by $k$ is the same as taking the sum over all faces of $\sigma$ as we did for our map in (\ref{eq: our_diff}). So up to a sign depending on only $p$, the map in (\ref{eq:brown_diff}) equals the map in (\ref{eq: our_diff}).

Our $E^1$ page consists of the following:
\[
0 \leftarrow E^1_{10} \xleftarrow{d^1_2} E^1_{20} \xleftarrow{d^1_3} E^1_{30} \xleftarrow{d^1_4} ...
\] 
We are only concerned with computing $E^1_{10}$ and $im(d^1_2)$ since 
\[E^1_{10}/im(d^1_2) = H_0(\Gamma, St \otimes_{\mathbb{Z}} \mathbb{Q}) \cong \mathbb{M}_2(\Gamma)\]
 by Corollary \ref{cor: k steinberg isom} and the definition of coinvariants.

\subsection{Stabilizer Groups of the Reduced Symbols}\label{subsec: stab groups}
We now focus on computing $E^1_{10}$ and $im(d^1_2)$, which means finding the relations among the reduced symbols. The first task is to find the stabilizer group of the $1$- and $2$-cells in $\overline\Gamma$.

The stabilizers of the cusps are as follows. Given a cusp $\alpha \in \mathbb{P}^1(F)$, we can identify this cusp with the infinite ray given by the set of vertices $(c(p,n))_{n\geq N}$ for $N \geq 1$. Let $\overline\Gamma_i \subset \overline\Gamma$ be the stabilizer of $c(p,i)$, which can be computed using \cite[Proposition 6, 9]{takahashi}. We have the following result.

\begin{lemma}\textup{\cite[Lemma 3.1]{mason_auto}}\textup{\cite[Theorem 4]{takahashi}}.
    Let $\overline\Gamma_{\alpha}$ denote the stabilizer of $\alpha$, then we have $\overline\Gamma_i \subseteq \overline\Gamma_{i+1}$ for $i \geq 1$ and $\overline\Gamma_{\alpha} = \bigcup_{i \geq 1} \overline\Gamma_i$.
\end{lemma}

We can use this result to compute the stabilizer group of the $1$-cells, which are the reduced symbols, and the $2$-cells. These groups will be finite.
\begin{lemma}\label{lem: stab red sym}
    Given a reduced symbol $[\alpha, \beta]$ with minimal vertex $v$, let $\overline\Gamma_{[\alpha,\beta]}$ denote the stabilizer of the symbol and $\overline\Gamma_v$ denote the stabilizer of its minimal vertex. Then we have $\overline\Gamma_{[\alpha,\beta]} \subseteq \overline\Gamma_v$.
\end{lemma}

\begin{proof}
Let $\{c(p,n)\}_{n\geq 1}$ and $\{c(p',n)\}_{n\geq 1}$ denote the rays corresponding to $\alpha$ and $\beta$ respectively. Let $\gamma \in \overline\Gamma_{[\alpha, \beta]}$ be a stabilizer of the reduced symbol $[\alpha, \beta]$. Then either $\gamma$ fixes both $\alpha$ and $\beta$ or it takes one to the other. 
    
Suppose $\gamma$ fixes both cusps. This means either $\gamma$ fixes all vertices on the path that the symbol $[\alpha, \beta]$ traces on $\mathcal{T}$ or $\gamma$ shifts every vertex by some finite number $k$. The latter means $\gamma \cdot c(p,n) = c(p,n+k)$. This is impossible because the vertices $c(p,n)$ and $c(p,n+k)$ are two different orbits under the action of $\overline\Gamma$. So $\gamma$ fixes every vertex, meaning $\gamma \in \overline\Gamma_v$. 

Now suppose $\gamma \cdot \alpha = \beta$ and $\gamma \cdot \beta = \alpha$, meaning the two cusps are equivalent modulo the action of $\overline\Gamma$. This means that $\gamma \cdot c(p,n) = c(p',n)$ for all $n \geq 1$. So they have the same label on $\mathcal{T}$.
This cannot happen if $[\alpha, \beta]$ is a symbol of type $o$ or $s$. If $v$ is an $e$-vertex, this means that the edge $[v, c(p,1)]$ gets taken to the edge $[v, c(p',1)]$. If $v$ is an $ns$-vertex with $\alpha$ and $\beta$ $\overline\Gamma$-equivalent, then the edge $[v,o]$ gets taken to the edge $[v,o']$ where $o, o'$ are adjacent vertices of $v$. This means that $\gamma$ acts non-trivially on the edges of $v$ while fixing $v$, meaning $\gamma \in \overline\Gamma_v$. 
\end{proof}

The $2$-cells are triples of cusps where any pair is a reduced symbols. A stabilizer of a $2$-cell acts as a permutation of the three cusps. There are at most 2 minimal vertices to consider for each $2$-cell. We describe the $2$-cells existing on $\mathcal{K}$ and use Lemma \ref{lem: stab red sym} to compute the stabilizer groups.
\begin{itemize}
\item Triples of cusps where every pair is a $e(p)$-symbol corresponding to the same rational point $p$ on $E$: The stabilizer group of any reduced symbol is isomorphic to ${\rm{GL}_2}(\mathbb{F}_q)$, which acts transitively on the edges of $e(p)$ and hence on the cusps. The stabilizer group of such a $2$-cell is contained in $\overline\Gamma_{e(p)} \cong {\rm{GL}_2}(\mathbb{F}_q)$.
\item Triples of cusps where every pair is an $o$-symbol or triples of cusps where two pairs are $o$-symbols and one pair is an $s$-symbol: In this case, all three cusps are orbit-representative cusps. There is no element in $\overline\Gamma$ that can take one cusp to another. The stabilizer groups of such $2$-cells are contained in $\overline\Gamma_{o} \cong \mathbb{F}_q^{\times}$.
\item Triples of cusps where every pair is a $ns$-symbol corresponding to the same minimal vertex $v(l)$: There are a couple cases to consider. If the three cusps are contained in the same orbit modulo the action of $\overline\Gamma$, then we can use $\overline\Gamma_{v(l)}$ to permute the cusps. Else, at least two of the cusps are in different orbits, in which case there is no element in $\overline\Gamma$ taking one cusp to the other. The stabilizer group of such $2$-cells is contained in $\overline\Gamma_{v(l)} \cong \mathbb{F}_{q^2}^{\times}$.
\item Triples of cusps where two pairs are $ns$-symbols corresponding to the same minimal vertex and one pair is an $o$- or an $s$-symbol: Let $\alpha, \beta$ be the cusps forming an $o$- or $s$-symbol, and let $\gamma$ be the other cusp. Without loss of generality, suppose $\alpha, \beta \in R$ and $\gamma \in R^i$ for $i \neq 0$. Clearly, the only permutation possible is switching $\alpha$ or $\beta$ with $\gamma$ while fixing the other. However, any element that permutes $\alpha$ and $\gamma$ does not fix $\beta$ but sends $\beta$ to the $\overline\Gamma$-equivalent cusp in $R^i$. The stabilizer group of such $2$-cells is contained in $\overline\Gamma_{o} \cong \mathbb{F}_q^{\times}$.
\end{itemize}


\subsection{Relations among the Reduced Symbols} 
Given an elliptic curve $E/\mathbb{F}_q$, we now give a group presentation of $\mathbb{M}_2(\Gamma)$ using the stabilizer group of the minimal vertices, $1$- and $2$-cells.

Recall that $R$ is the set of orbit-representative cusps. Let $g$ be a coset representative of $\Gamma\backslash\overline\Gamma$. Let $g \cdot [\alpha, \beta] = [g \cdot \alpha, g \cdot \beta]$ where the action of the right hand side of the equation is given by fractional linear transformations. 

We restate Proposition \ref{thm: explicit red sym} for a congruence subgroup $\Gamma$ of $\overline\Gamma$ to describe the generators. Since there are only finitely many coset representatives of $\Gamma\backslash\overline\Gamma$, the number of generators are finite modulo the action of $\Gamma$.
\begin{corollary}\label{cor: exp red sym cong}
The reduced symbols over $\Gamma$ are given by the following.
\begin{itemize}
	\item The $e(\infty)$-symbols are given by $g \cdot [\infty, 0]$,
	\item The $e(p)$-symbols are given by $g \cdot [(y-m)/(x-l), S_p \cdot (y-m)/(x-l)]$,
	\item The $o$- and $s$-symbols are given by $g \cdot [\alpha, \beta]$ where $\alpha, \beta \in R$,
	\item The $ns$-symbols are $\overline\Gamma$-translates of $[\alpha, \beta]$ where $\alpha \in R$, $\beta \in R_{l'}^i$, and $v(l')$ is an $ns$-vertex.
\end{itemize}
\end{corollary}

It now remains to use the results above to find the relations among the reduced symbols in order to compute $E^1_{10}$ and $im(d^1_2)$.

First, consider the relations between $e(\infty)$-symbols. Any reduced symbol of type $e(\infty)$ is $\overline\Gamma$-equivalent to $[\infty, 0]$ and any $2$-cell is $\overline\Gamma$-equivalent to the $2$-cell given by $\{\infty, 0,1\}$. Letting $[g] = g \cdot [\infty, 0]$, we have the following set of relations.
\begin{equation}\label{e infty relations}
[g] - [g \cdot D]= 0, [g] + [g \cdot S] = 0, [g] + [g \cdot T]  + [g \cdot T^2] = 0.
\end{equation}
where $D$ is a diagonal matrix in ${\rm{GL}_2}(\mathbb{F}_q)$,  
    	 $S = 
	\begin{pmatrix}
	0 & 1 \\
	-1 & 0
	\end{pmatrix}, T = 
	\begin{pmatrix}
	0 & 1 \\
	-1 & 1
	\end{pmatrix}$.

Now, consider the $e(p)$-symbols where $p = (l,m)$ is a rational point on $E$ but not the point at $\infty$. There is an isomorphism $\phi: \overline\Gamma_{v_{e(p)}} \to {\rm{GL}_2}(\mathbb{F}_q)$ given by sending $x$ to $l$ and $y$ to $m$. Like the $e(\infty)$-case, there is one reduced symbol and one $2$-cell modulo the action of $\overline\Gamma$. Letting $[g] = g \cdot [(y-m)/(x-l), S_p \cdot (y-m)/(x-l)]$ where $[(y-m)/(x-l), S_p \cdot (y-m)/(x-l)]$ is the $e(p)$-symbol found using Proposition \ref{thm: explicit red sym}, we have the following set of relations.
\begin{equation}\label{e non infty relations}
[g] - [g \cdot D'] = 0, [g] + [g \cdot S'] = 0, [g] + [g \cdot T'] + [g \cdot T'^2] = 0.
\end{equation}
where $D' = \phi^{-1}(D), S' = \phi^{-1}(S), T' = \phi^{-1}(T)$.

Given $\alpha, \beta, \delta \in R$, any pair is an $o$- or $s$-symbol. No cusps in $R$ are equivalent modulo the action of $\overline\Gamma$, meaning we need all the reduced symbols and $2$-cells formed by the cusps in $R$. This gives us the the following set of relations.
\begin{equation}\label{o s relations}
g \cdot [\alpha, \beta] + g \cdot [\beta, \alpha] = 0, g \cdot [\alpha, \beta] + g \cdot [\beta, \delta] + g \cdot [\delta, \alpha] = 0.
\end{equation} 

Finally, for $ns$-symbols with minimal vertex $v(l')$, let $P_{l'}$ be a generator of $\overline\Gamma_{v_{l'}}$. Let $R$ be the set of orbit-representative cusps denoted by $\{\alpha_1, \dots, \alpha_{q+1}\}$ and $R_{l'}^i = \{P_{l'}^i \cdot \alpha_j : \alpha_j \in R\}$ for $i \in I = \{0,1, 2, ... q\}$. Let $\alpha_j^i$ denote the cusp $\alpha_j \in R_{l'}^i$ and let $[g]_{j_1, j_2}^{i_1, i_2}$ denote the symbol $g \cdot [\alpha^{i_1}_{j_1}, \alpha^{i_2}_{j_2}]$. Note that if $i_1 = i_2$, then the symbol is an $o$- or $s$-symbol.  Using powers of $P_{l'}$, $[\alpha^{0}_{j_1}, \alpha^{i_2}_{j_2}]$ for $i_2 \in \{1, 2, ..., q\}$ gives us a complete set of reduced symbols modulo $\overline\Gamma$. The element $P_{l'}^k$ sends the cusp $\alpha_i^j$ to $\alpha^{i'}_j$ where $i' < q+1$ such that $i' \equiv i+k \mod q+1$. This means on a reduced symbol, $P_{l'}^k$ ``rotates" the reduced symbol around $v(l')$ by $k$-steps. Using this action on the reduced symbols and the $2$-cells shows us the set of relations we need are as follows. 

First, assume that $q$ is odd. Then we need the relations given by the image of $d_2^1$ on the $2$-faces of the triples of cusps below. The remaining $2$-faces are $\overline\Gamma$-equivalent by some power of $P_{l'}$.
\begin{equation*}
[g]_{j_1,j_2}^{0, i} + [g \cdot P_{l'}^{i}]_{j_2, j_1}^{0,q+1-i} = 0 \text{ for } 1 \leq i \leq (q+1)/2,
\end{equation*}
\begin{equation*}
[g]_{j_1,j_2}^{0, i_1} + [g \cdot P_{l'}^{i_1}]_{j_2, j_3}^{0,i_2-i_1} + [g \cdot P_{l'}^{i_2}]_{j_3, j_1}^{0, q+1-i_2} = 0 \text{ for } 1 \leq i_1 < i_2 \leq (q+1)/2,
\end{equation*}
\begin{equation}\label{ns rel}
[g]_{j_1, j_2}^{0,i} + [g]_{j_2, j_3}^{i,i} + [g \cdot P_{l'}^{i}]_{j_3, j_1}^{0, q+1-i} = 0 \text{ for } 1 \leq i \leq (q+1)/2.
\end{equation}
If $q$ is even, we change the inequalities stated at the end of each relation above from $(q+1)/2$ to $q/2$. 

We are now able to state Theorem \ref{main 2} in a way that is suitable for explicit computations.
\begin{theorem}[Restatement of Theorem \ref{main 2}]\label{main result}
Let $\Gamma$ be a congruence subgroup of $\overline\Gamma$. The group of modular symbols over $\Gamma$, denoted $\mathbb{M}_2(\Gamma)$, is given by the set of generators stated in Corollary \ref{cor: exp red sym cong} and the set of relations (\ref{e infty relations}) to (\ref{ns rel}).
\end{theorem}

\begin{example}\label{ex: example relation}
	\normalfont We return to the elliptic curve $E/\mathbb{F}_3$ given in Example \ref{ex: example red sym}. The curve is given by the equation $y^2=x^3+x-1$ and the set of orbit-representative cusps is $R = \{\infty, y/(x+1), (y-1)/(x-1), (y+1)/(x-1)\}$. 
	
Let $g$ be a coset representative of $\Gamma\backslash\overline\Gamma$. Let $g \cdot [\alpha, \beta] = [g \cdot \alpha, g \cdot \beta]$ where the action of the right hand side of the equation is given by fractional linear transformations. Let $[g] = g \cdot [\infty, 0]$. We have the following set of relations of $e(\infty)$-symbols. 
\begin{itemize}
\item $[g] - [g \cdot D]= 0$, where $D = 
    	\begin{pmatrix}
        \pm 1 & 0 \\
        0 & \mp 1
    	\end{pmatrix}$, 
    
\item $[g] + [g \cdot S] = 0$, where $S = 
	\begin{pmatrix}
	0 & 1 \\
	-1 & 0
	\end{pmatrix}$,

\item $[g] + [g \cdot T]  + [g \cdot T^2] = 0$, where $T = 
	\begin{pmatrix}
	0 & 1 \\
	-1 & 1
	\end{pmatrix}$.
\end{itemize}
	
Now, let $[g]$ denote the $e((-1,0))$-symbol $g \cdot [y/(x+1), (x^2-x+1)/y]$. We have the following set of relations of $e((-1,0))$-symbols. 
\begin{itemize}
\item $[g] - [g \cdot D'] = 0$, where $D' = 
	\begin{pmatrix}
        \pm (x-x^2) & (x+1)y \\
        -y & \mp (x-x^2)
        \end{pmatrix}$,
\item $[g] + [g \cdot S'] = 0$, where $S' = 
	\begin{pmatrix}
        (1-x)y & x^3+x^2-x \\
        -x^2 & (x-1)y
        \end{pmatrix}$, 
\item $[g] + [g \cdot T'] + [g \cdot T'^2] = 0$, \\
where $T' = 
	\begin{pmatrix}
	(1-x)y - x^2 + x - 1 & (x+1)y + x^3 + x^2 - x \\
        -y-x^2 & (x-1)y + x^2-x-1
        \end{pmatrix}$.
\end{itemize}

Given $\alpha, \beta, \delta \in R$, we have the following set of relations of $o$- and $s$-symbols. 
\begin{itemize}
\item  $g \cdot [\alpha, \beta] + g \cdot [\beta, \alpha] = 0$,
\item $g \cdot [\alpha, \beta] + g \cdot [\beta, \delta] + g \cdot [\delta, \alpha] = 0$.
\end{itemize} 

Recall $P= \begin{pmatrix}
y+1 & 1-x^2 \\
x & -y+1 
\end{pmatrix}$ is the generator of $\Gamma_{v(0)}$. Let $\{\alpha_1, ... \alpha_4\}$ denote the cusps in $R$ and $R^i = \{P^i \cdot \alpha_j : \alpha_j \in R\}$ for $i \in I = \{0, 1, 2, 3\}$. Let $\alpha_j^i$ denote the cusp $\alpha_j \in R^i$ and let $[g]_{j_1, j_2}^{i_1, i_2}$ denote the symbol $g \cdot [\alpha^{i_1}_{j_1}, \alpha^{i_2}_{j_2}]$. Note that if $i_1 = i_2$, then the symbol is an $o$- or $s$-symbol. We have the following set of $ns$-symbols.

\begin{itemize}
\item $[g]_{j_1, j_2}^{0, i} + [g \cdot P^i]_{j_2, j_1}^{0,4-i} = 0$ for $i = 1, 2, 3$,
\item $[g]_{j_1, j_2}^{0, 1} + [g \cdot P]_{j_2, j_3}^{0,1} + [g \cdot P^2]_{j_3, j_1}^{0,2} = 0$,
\item $[g]_{j_1, j_2}^{0,i} + [g]_{j_2, j_3}^{i,i} + [g \cdot P^i]_{j_3, j_1}^{0,4-i} = 0$ for $i = 1, 2, 3$.
\end{itemize}
\end{example}

\section{Resolution of the Steinberg Module}\label{sec: res stein}
In this section, we finish the proof of Theorem \ref{thm: k resolution}. We need to show that the kernel of $\epsilon$ in Equation (\ref{eq: k res}) is equal to the image of $\partial_2$. The kernel of $\epsilon$ contains all formal sums of $1$-cells of $\mathcal{K}$ whose classes map to 0 in the Steinberg module. 

A relation between modular symbols is an equation of the form 
\[\sum_{i=1}^n m_i [\alpha_i, \beta_i] = 0 \text{ where } m_i \in \mathbb{Q}.\]
By Theorem \ref{uni sum}, we can reduce this to an equation where all the $[\alpha_i, \beta_i]$ are reduced symbols. 
A $1$-cell of $\mathcal{K}$ is a reduced symbol. This means any relation of reduced symbols $\sum_{i=1}^n m_i [\alpha_i, \beta_i] = 0$ is in the kernel of $\epsilon$. For any minimal vertex $v$, there are a finite number of cusps attached to $v$, which are the vertices of the $v$-complex. Any pair of cusps $\alpha$ and  $\beta$ gives a reduced symbol $[\alpha, \beta]$,which is a $1$-cell on the $v$-complex. It satisfies the \textit{two-term relation}: $[\alpha, \beta] + [\beta, \alpha] = 0$. Any triple of cusps $\alpha, \beta, \gamma$ satisfies the \textit{three-term relations}: $[\alpha, \beta] + [\beta, \gamma] + [\gamma, \alpha]= 0$. This is the image of $\partial_2$ on the $2$-cells of the $v$-complex. We call these relations the \textit{two-term and three-term relations over the minimal vertex $v$}. The three-term relations over each type of minimal vertex are given in Section \ref{subsec: stab groups}. 
\begin{itemize}
\item Over a single $e(p)$-vertex, there are $q+1$ cusps, whose pairs form reduced symbols of type $e(p)$. The three-term relation consists of $[\alpha, \beta] + [\beta, \gamma] + [\gamma, \alpha] = 0$. 
\item Over a single $o$-vertex, there are $q+1$ cusps, whose pairs form reduced symbols of type $o$ or type $s$. The three-term relation consists of $[\alpha, \beta] + [\beta, \gamma] + [\gamma, \alpha] = 0$ where at most one reduced symbol is type $o$.
\item Over a single $ns$-vertex, there are $(q+1)^2$ cusps, whose pairs form reduced symbols of type $ns$, type $o$, or type $s$. The three-term relation consists of $[\alpha, \beta] + [\beta, \gamma] + [\gamma, \alpha] = 0$ where at most one reduced symbol is type $o$ or type $s$.
\end{itemize}
Finishing Theorem \ref{thm: k resolution} amounts to proving the following.
\begin{theorem}\label{thm: relation}
Let $\sum_{i=1}^n m_i [\alpha_i, \beta_i] = 0$ be a relation of reduced symbols and let $v_1, v_2, \dots, v_k$ be the set of unique minimal vertices of the reduced symbols. Then the $\sum_{i=1}^n m_i [\alpha_i, \beta_i]$ can be written as a sum of the two-term relations and three-term relations over the $v_j$.
\end{theorem}

\subsection{Balanced Edges}
To prove this, we return to looking at modular symbols over $\mathcal{T}$ instead of their analog on $\mathcal{K}$. 
For an edge $[v,w]$, a nonzero modular symbol interpreted as a path can go from $v$ to $w$ or from $w$ to $v$ through this edge. We say that a symbol \textit{travels} on the edge $[v,w]$ when this happens.
\begin{definition}
An edge $[v, w]$ is \textbf{balanced} if the total coefficient $a$ of the symbols traveling from $v$ to $w$ equals the total coefficient $b$ of the symbols traveling from $w$ to $v$ i.e. $a=b$. Similarly, we say a cusp $\alpha$ is \textbf{balanced} if the total coefficient of the symbols entering $\alpha$ equal the total coefficient of symbols exiting $\alpha$.
\end{definition}

We first need the following results about modular symbols.
\begin{lemma}\label{all edges balanced}
   Given a relation of modular symbols, every edge that the modular symbol travel on is balanced.     
\end{lemma}

Definition \ref{def: mod sym} tells us that any modular symbols satisfy the following relations:
\begin{equation}\label{rel: all mod three-term}
    [\alpha, \beta] + [\beta, \alpha] = 0,
    [\alpha, \beta] + [\beta, \gamma] + [\gamma, \alpha] = 0. 
\end{equation}

We first prove the following result that applies to all modular symbols including the non-reduced symbols. To ease the notation of a symbol, we will notate a modular symbol $[\alpha, \beta]$ as $\tau$ with coefficient $m_{\tau}$ throughout this section. This notation allows us to denote the coefficient of $-\tau$ as $m_{-\tau}$ and write a relation as $0 = \sum_{i=1}^n m_{i} [\alpha_i, \beta_i] = \sum_{i=1}^n m_{\tau_i} \tau_i$. 

\begin{theorem}\label{thm: all mod satisfy rel}
    A relation of modular symbols can be written as a sum of the two-term and three-term relations given in (\ref{rel: all mod three-term}).
\end{theorem}

\begin{proof}
Given a relation of modular symbols $\sum_{i=1}^n m_{\tau_i} \tau_i = 0$, we reduce the sum $\sum_{i=1}^n |m_{\tau_i}|$ to zero modulo the relations. 

Use the two-term relation so that for any modular symbol $\tau$ with $m_{\tau} \neq 0$, we have $m_{-\tau} = 0$. 
    
Choose a cusp $\alpha$. There is some symbol exiting $\alpha$ and entering another cusp $\beta$. Let this symbol be $\tau$. Since every cusp is balanced, there is some symbol leaving $\beta$ and entering another cusp $\gamma$ with $\alpha \neq \gamma$ since $n_{-\tau} = 0$. Finally, consider the symbol from $\gamma$ to $\alpha$. The sum $[\alpha, \beta] + [\beta, \gamma] + [\gamma, \alpha]$ is equal to zero.

Then we have 
    \begin{gather*}
    0 = \sum_{i=1}^n m_{\tau_i} \tau_i = \sum_{i=1}^n m_{\tau_i} \tau_i - [\alpha, \beta] - [\beta, \gamma] - [\gamma, \alpha] \\
    = \sum_{i=1}^n m_{\tau_i} \tau_i - [\alpha, \beta] - [\beta, \gamma] + [\alpha, \gamma].
    \end{gather*}
    We have removed two symbols at the expense of adding one with every cusp still balanced. This means that $\sum_{i=1}^n |m_{\tau_i}|$ was reduced by one. Since our sum $\sum_{i=1}^n |m_{\tau_i}|$ is finite, we can continue in this way until it is zero modulo the relations
    \end{proof}

We will use this proof method in the subsequence subsections. Note that Theorem \ref{thm: all mod satisfy rel} is a result about all modular symbols including non-reduced symbols. This is not a proof of Theorem \ref{thm: relation}. We need Theorem \ref{thm: all mod satisfy rel} to prove Lemma \ref{all edges balanced}, which is needed in the subsequence sections.

\begin{proof}[Proof of Lemma \ref{all edges balanced}] 
Given a relation of symbols that is of either form given in (\ref{rel: all mod three-term}), the edges of $\mathcal{T}$ that the symbols travel on are balanced. The desired result follows immediately since we have shown in Theorem \ref{thm: all mod satisfy rel} that any relation of modular symbols is a sum of the two-term and three-term relations given in (\ref{rel: all mod three-term}).
\end{proof}

\subsection{Interactions between Minimal Vertices}
From this section onwards, a relation of symbols is always a relation of reduced symbols. It now remains to prove Theorem \ref{thm: relation}. 
Given a relation of symbols $\sum_{i=1}^n m_i [\alpha_i, \beta_i] = 0$, let $\mathcal{V}$ be the set of minimal vertices that identify the symbols. Suppose we could write the relation as a sum of relations where each relation involves only symbols corresponding to one minimal vertex $v \in \mathcal{V}$. That is, we have the following:
\[
    0 = \sum_{i=1}^n m_i [\alpha_i, \beta_i] = \sum_{v \in \mathcal{V}} \left(\sum_{j=1}^{n_v} m_{j,v} [\alpha_j, \beta_j]_v \right)
\]
such that $[\alpha_j, \beta_j]_v$ is a $v$-symbol and $\sum_{j=1}^{n_v} m_{j,v} [\alpha_j, \beta_j]_v = 0$ for all $v \in \mathcal{V}$. 
We say that the relation \textit{decomposes over the minimal vertices}. However, it is clear that this cannot happen since certain three-term relations contain two different types of symbols. Each minimal vertex has a Serre invariant $N(E)$. Since each minimal vertex has corresponding reduced symbols, we can associate the invariant to the symbols as well. The two cusps that form a $v$-symbol with Serre invariant $n$ could be cusps that are attached to multiple minimal vertices with invariant strictly less than $n$. For example, the two cusps forming an $o$-symbol could be cusps attached to multiple $ns$-vertices. This means over an $ns$-vertex $v$, the sum of $ns$-symbols might not form a relation. We may need reduced symbols corresponding to the $o$- and $s$-vertices adjacent to and two degrees away from $v$ to form a relation. We address this issue by considering the different ways that minimal vertices interact as addressed in Section \ref{sec: simplicial k}. 

We prove Theorem \ref{thm: relation} for three cases where the symbols satisfy a certain condition on the invariant. In each case, we deal with the issue of decomposing over the minimal vertices. Combining these results will prove Theorem \ref{thm: relation}.

\subsection{Symbols with $N(E) \geq 0$}
Suppose we are given a relation consisting of only $e(p)$ and $s$-symbols. From Section \ref{sec: simplicial k} and Figure \ref{fig: quot_to_large}, we have the following possible interactions. There are no interactions between two $e$-vertices or two $s$-vertices with different labels or an $e(p)$-vertex and a $s$-vertex. A set of $e$-vertices or a set of $s$-vertices with the same labels $e(p)$ or $v(l)$ respectively can interact trivially.

\begin{lemma}\label{lem: triv int 0}
Given a relation of symbols $\sum_{i=1}^n m_i [\alpha_i, \beta_i] = 0$, let $\mathcal{V}$ denote the set of minimal vertices that identify the symbols in our relation. Furthermore, suppose all $v_i \in \mathcal{V}$ have the same label, either $e(p)$ or $v(l)$ that interact trivially pairwise at a cusp $\alpha$. Such a relation can be written as the following. 
\[0 = \sum_{i=1}^n m_i [\alpha_i, \beta_i] = \sum_{v \in \mathcal{V}} \left(\sum_{j=1}^{n_{v}} m_{j,{v}} [\alpha_j, \beta_j]_{v} \right)\]
where $\sum_{j=1}^{n_v} m_{j,v} [\alpha_j, \beta_j]_v = 0$ for all $v \in \mathcal{V}$.
\end{lemma}
\begin{proof}
We first prove this when $\mathcal{V}$ has two elements. Suppose that the two vertices have the same label $e(p)$. The cusp $\alpha$ is only the cusp that can appear in both $v$- and $w$-symbols. Let $c(p_v,r)_{r \geq 1}$ and $c(p_w,r)_{r \geq 1}$ denote the rational ends adjacent to $v$ and $w$ respectively that correspond to $\alpha$. There is some $r' \geq 2$ such that for all $r'' > r'$, $c(p_v, r'') = c(p_w, r'')$. Call the vertex $c(p_v, r') = c(p_w, r')$ the \textit{initial common vertex} and $r'$ the \textit{identifying number}. There are both $v$- and $w$-symbols traveling on the edges between these vertices. Suppose that the total coefficients of symbols entering the cusp $\alpha$ is $a > 0$. Since we have a relation, it follows that the total coefficients of symbols exiting $\alpha$ is $-a$. For $r'' < r'$, only $v$- and $w$-symbols travel on the edges $[c(p_v, r'-1), c(p_v, r')]$ and $[c(p_w, r'-1), c(p_w, r')]$ respectively. call these edges the \textit{branching edges}. Every edge is balanced by \ref{all edges balanced}. Let $b > 0$ and $c > 0$ be the total sum of the coefficient of symbols traveling towards $\alpha$. It follows that $a = b + c$. The total coefficients of $v$-symbols entering $\alpha$ is $b$ and exiting $\alpha$ is $-b$. We have a similar result for $w$-symbols and $c>0$. This means the set of $v$-symbols and the set of $w$-symbols each form a relation. The case for when the two vertices are both type $s$ with the same label is the same argument as the above except that $r' \geq 1$. Induction on the number of elements in $\mathcal{V}$ and modifying the proof above gives the desired result. Each pair of vertices has an initial common vertex with an identifying number. The initial common vertex we need is the one with the largest identifying number. 
\end{proof}

We say that the relation \textit{decomposes over each minimal vertex} when Lemma \ref{lem: triv int 0} is satisfied. For each minimal vertex $v$, the set of $v$-symbols form a relation. Any relation involving symbols with Serre invariant $N(E) > 0$ will decompose over each minimal vertex, which is either type $e(p)$ or type $s$. The only possible relation over a single $s$-vertex is a constant multiple of the two-relation relation for $s$-symbols. What is left is to prove the following:

\begin{lemma}\label{e sum of relations}
A relation of symbols over one $e(p)$-vertex is a sum of the two-term and three-term relations over the $e(p)$-vertex.
\end{lemma}

\begin{proof}
The strategy is the same as Theorem \ref{thm: all mod satisfy rel}. The relation involves only the $q+1$ cusps attached to the $e(p)$-vertex and any pair or triple of cusps gives a valid two-term or three-term relation. Hence, a similar argument to the proof of Theorem \ref{thm: all mod satisfy rel} gives the desired result.
\end{proof}

\subsection{Symbols with $N(E) \geq -1$}
The interactions are as follows. We know $e(p)$-vertices could interact trivially with $o$-vertices. A set of $o$-vertices could interact with each other trivially or non-trivially depending on the cusp. We state a result similar to Lemma \ref{lem: triv int 0}.

\begin{lemma}\label{lem: triv int -1}
Given a relation of symbols $\sum_{i=1}^n m_i [\alpha_i, \beta_i] = 0$, let $\mathcal{V}$ denote the set of minimal vertices that identify the symbols in our relation. Furthermore, suppose all $v_i \in \mathcal{V}$ consists of $e(p)$-vertices with the same label or $o$-vertices that interact trivially pairwise at a cusp $\alpha$. Such a relation can be written as the following. 
\[0 = \sum_{i=1}^n m_i [\alpha_i, \beta_i] = \sum_{v \in \mathcal{V}} \left(\sum_{j=1}^{n_{v}} m_{j,{v}} [\alpha_j, \beta_j]_{v} \right)\]
where we allow $[\alpha_j, \beta_j]_v$ to be an $s$-symbol if $v$ is type $o$ and $\sum_{j=1}^{n_v} m_{j,v} [\alpha_j, \beta_j]_v = 0$ for all $v \in \mathcal{V}$.
\end{lemma}

\begin{proof}
In this case, we need to first prove the following two cases: either the two minimal vertices interacting trivially are types $e(p)$ and $o$ or the two vertices are both type $o$. A similar argument to the proof of Lemma \ref{lem: triv int 0} suffices. The initial common vertex could be a vertex of label $v(l)$ where $l \in \mathbb{P}^1(\mathbb{F}_q)$ and we let the identifying number take nonpositive values. Furthermore, a branching edge on which $o$-symbols travel could also have $s$-symbols traveling.
\end{proof}

This means that given a relation of symbols with $N(E) \geq -1$, over any minimal vertex of type $e(p)$, the set of the $e(p)$-symbols form a relation.
This allows us to consider a relation involving only $s$- and $o$-symbols and non-trivial interactions between the $o$-vertices. 
An $o$-vertex $O$ has $q+1$ cusps attached. The vertex $O$ sees a total of $\binom{(q+1)}{2}$ symbols of type $s$ and $o$. We do need to be careful since not every pair of cusps is an $o$-symbol.
\begin{lemma}\label{result for one o}
A relation of symbols involving exactly one $o$-vertex with its adjacent $s$-vertices is a sum of the two-term relation and three-term relations over the $o$-vertex and the two-term relations over the adjacent $s$-vertices.
\end{lemma}

\begin{proof}
Again, the strategy is the same as Theorem \ref{thm: all mod satisfy rel}. The relation involves only the $q+1$ cusps attached to the $o$-vertex and its adjacent $s$-vertices. Any pair or triple of cusps gives a valid two-term or three-term relation, and by Lemma \ref{all edges balanced}, every edge is balanced. In particular, the edges of the $o$-vertex are balanced. This means a similar argument to the proof of Theorem \ref{thm: all mod satisfy rel} gives the desired result.
\end{proof}

Two $o$-vertices can interact nontrivially through the cusps of the $s$-vertex. (See Figure \ref{fig: multiple_o} for an example). We need to show that we can still decompose our relation over all the $o$-vertices. 

\begin{figure}[htb]
\centering
\includegraphics[width=7cm, height=6cm]{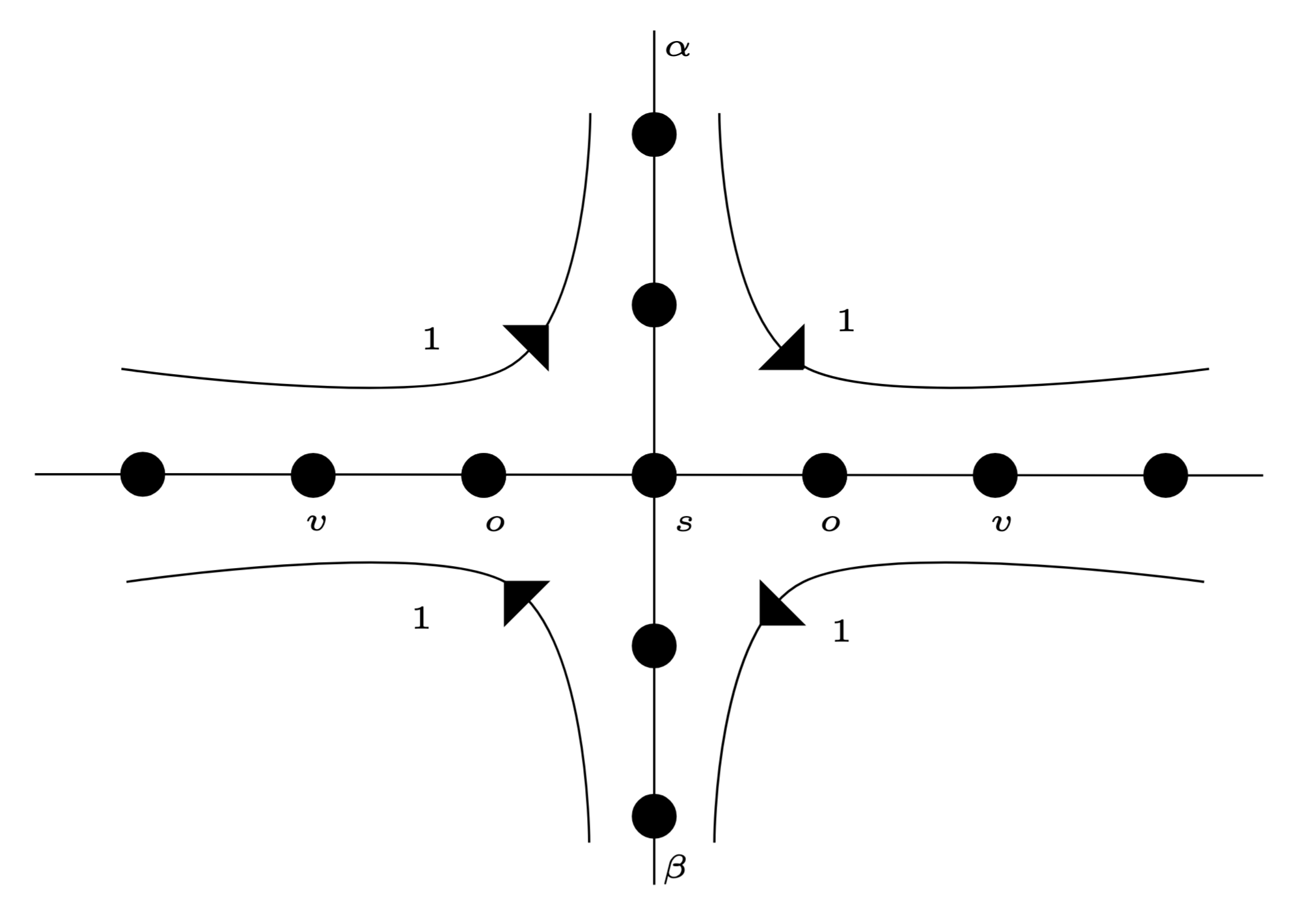}
\caption{Two $o$-vertices interacting nontrivally due to a shared $s$-vertex.}\label{fig: multiple_o}
\end{figure}

\begin{lemma}\label{result for multiple o}
 A relation of symbols involving only $o$ and $s$-symbols is a sum of the two-term and three-term relations over the $o$-vertices and the two-term relations over the adjacent $s$-vertices.
\end{lemma}

\begin{proof}
The idea is to apply Lemma \ref{result for one o} to the $o$-vertices. We first prove we can do this over two o-vertices $O \text{ and } O'$. Suppose they interact nontrivially through a $s$-vertex $S$ with cusps $\alpha$ and $\beta$.

Use the two-term relations over the $o$- and $s$-vertices so that all $O$-symbols involving $\alpha$ are written as entering $\alpha$ and $O$-symbols involving $\beta$ are written as leaving $\beta$. Write all $O'$-symbols as leaving $\alpha$ and entering $\beta$. This also ensures that if $n_{\tau} \neq 0$, then $n_{-\tau} = 0$. Without loss of generality, assume the total contribution of the coefficients of the $O$-symbols is $a>0$. Since the edge $[O,S]$ is balanced, the contribution of $O$-symbols to $\beta$ is $-a$. Assume the contribution of the $S$-symbol to $\beta$ is $a'>0$ and $-a'$ to $\alpha$. That is, the symbol $[\alpha, \beta]$ has coefficient $a'>0$. This means the total contribution of $O'$-symbols to $\alpha$ is $a-a'$ and $-(a-a')$ to $\beta$ since the edge $[O',S]$ is balanced.

Add the symbols $a[\alpha, \beta] + a[\beta,\alpha]$ to our sum of symbols. The total contribution of $O$-symbols and the symbol $a[\alpha, \beta]$ to the cusp $\alpha$ is $a-a = 0$. Similarly, the total contribution to $\beta$ is $0$. The total contribution of $O'$-symbols and the symbol $(a-a')[\beta, \alpha]$ to the cusp $\alpha$ is $(a-a')-(a-a')= 0$ and similarly, $0$ to $\beta$. We can now apply Lemma \ref{result for one o} to $O$ and $O'$.

If we see multiple $o$-vertices and $s$-vertices interacting, we can construct a finite tree where the vertices are the $o$-vertices and there is an edge if there is an interaction. There are finitely many $o$-vertices in our relation so the tree must have leaves. At any leaf $O$, we can add the necessary $s$-symbols using the argument above to apply Lemma \ref{result for one o} to $O$, eliminating $O$ from our tree. Continuing in this way and induction on the number of vertices gives us the desired result. 
\end{proof}

\subsection{Symbols with $N(E) \geq -2$}
Now consider the case where all four reduced symbols are involved in a given relation. The interactions that need to be addressed are as follows.
An $e(p)$-vertex can interact trivially an $ns$-vertex. There can be both trivial and non-trivial interactions with $ns$-vertices all with different labels, but no interactions if they all have the same label.
A result and proof similar to Lemma \ref{lem: triv int -1} applied to the case of $e$- and $ns$-vertices allows us to consider a relation only involving $s$-, $o$- and $ns$-symbols and non-trivial interactions between $ns$-vertices for the remainder of this section.

An $ns$-vertex $N$ has $q+1$ $o$-vertices with $q+1$ cusps attached, giving $N$ a total of $(q+1)^2$ cusps. The vertex $N$ sees a total of $\binom{(q+1)^2}{2}$ symbols of type $s, o, \text{ and } ns$. We call the $o$ and $s$-vertices adjacent or two edges away from $N$ as the \textit{corresponding} $o$ and $s$ vertices of $ns$. The cusps attached to this $ns$-vertex are also the cusps attached to the corresponding $o$ and $s$-vertices. 

We first consider a relation with only one $ns$-vertex.
\begin{lemma}\label{result for one v not square}
 A relation of symbols over exactly one $ns$-vertex and its adjacent $o$-vertices and their adjacent $s$-vertices is a sum of the two-term and three-term relations over the $ns$-vertex and its corresponding $s$- and $o$-vertices.
\end{lemma}

\begin{proof}
The strategy is the same as Theorem \ref{thm: all mod satisfy rel} and the desired result follows by the similar reasons given in the proof of Lemma \ref{result for one o} applied appropriately.
\end{proof}

\begin{figure}[htb]
\centering
\includegraphics[width=7cm, height=6cm]{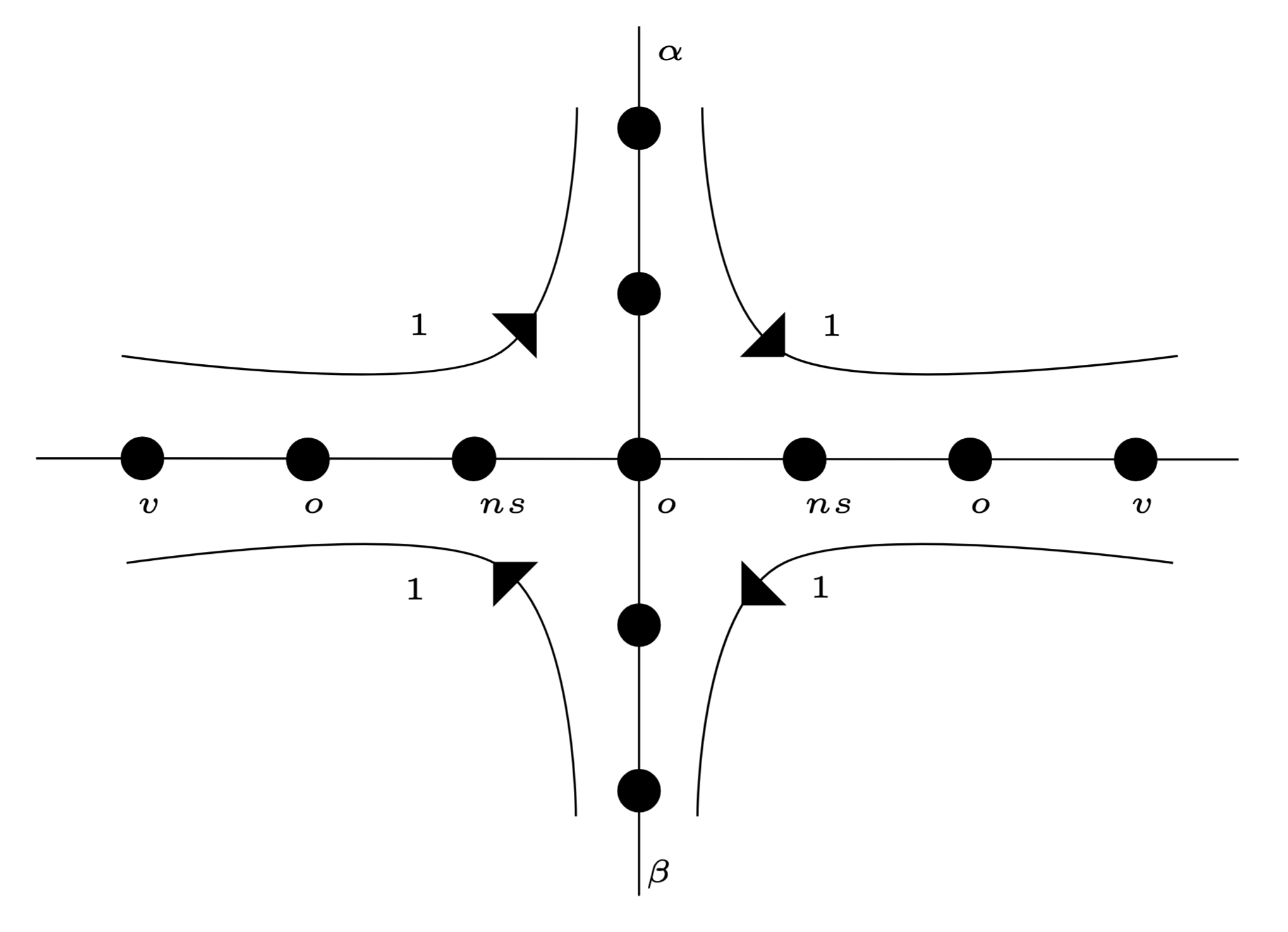}
\caption{Two $ns$-vertices interacting nontrivially due to a shared $o$-vertex.}
\label{fig: multiple_v_non1}
\end{figure}

There can be nontrivial interactions between $ns$-vertices through an $o$ or $s$-vertex. (See Figure \ref{fig: multiple_v_non1} for an example). We stat and prove a similar to Lemma \ref{result for one v not square}.

\begin{lemma}\label{result for many v non square}
A relation of $s$, $o$ and $ns$-symbols is a sum of the two-term and three-term relations over the $ns$-vertices and their corresponding $s$- and $o$-vertices.
\end{lemma}
\begin{proof}
    Like Lemma \ref{result for one v not square}, we want to apply Lemma \ref{result for one v not square} to the $ns$-vertices. We first prove this for two $ns$-vertices $N$ and $N'$ interacting nontrivially through a shared $o$-vertex $O$.     
First, use the two-term relations to simplify our relation. If $n_{\tau} \neq 0$, we want $n_{-\tau} = 0$. 

    We first consider the vertex $N$ and its $ns$-symbols. Like the proof of Lemma \ref{result for one v not square}, we can identify cusps $\alpha_1, \dots, \alpha_n$ attached to $O$ such that the total contribution of $N$-symbols to $\alpha_i$ is positive. If necessary, use the two-term relation so that all $N$-symbols enter $\alpha_i$. The total contribution of the coefficients of symbols entering $\alpha_i$ is $a_i > 0$. Likewise, $\beta_1, \dots, \beta_m$ are the cusps where all $N$-symbols leave $\beta_j$ with contribution $b_j < 0$. Every edge including $[N,O]$ is balanced so $\sum_{i=1}^n a_i = -\sum_{j=1}^m b_j$. For each $\alpha_i$, add the symbols $a_i[\alpha_i, \beta_j] + a_i[\beta_j, \alpha_i]$ for $j$ such that $|b_j| \geq a_i$. This ensures the net contribution of the $N$-symbols entering $\alpha_i$ and the symbol $a_i[\alpha_i, \beta_j]$ to the cusp $\alpha_j$ is $a_i - a_i = 0$ and that $b_j + a_i \leq 0$. When we have gone through every $\alpha_i$, the net contribution of the $N$-symbols and the symbols $a_i[\alpha_i, \beta_j]$ to every cusp $\alpha_i, \beta_j$ is zero. This means the $N$-symbols and the symbols $a_i[\alpha_i, \beta_j]$ form a relation.
    
    From the original relation, we only considered the $N$-symbols. The remaining original symbols of type $N'$, $o$ and $s$ along with the added symbols $a_i[\beta_j, \alpha_i]$ must form a relation. So we have decomposed our relation over $N$ and $N'$.

    If we multiple $ns$-vertices interacting, we construct a finite tree where $ns$-vertices are the vertices and there is an edge if there is an interaction. This tree must have at least one leaf. Start at a leaf $N$ and use the argument given above and one similar to the proof of Lemma \ref{result for one v not square} to delete $N$. Continue eliminating the leaves. At worst, we are left with a relation of $o$ and $s$-symbols, which we know how to deal with.
\end{proof}

\subsection{Proof of Theorem \ref{thm: relation}}
We now prove Theorem \ref{thm: relation}. 
\begin{proof}
The strategy is similar to Theorem \ref{thm: all mod satisfy rel}. Rewriting the relation as $\sum_{i=1}^n m_{\tau_i} \tau_i = 0$, we reduce the sum $\sum_{i=1}^n |m_{\tau_i}|$ to zero modulo the two-term and three-term relations over the minimal vertices.

We know that the symbols corresponding to a single $e(p)$-vertex each form a relation. We use the two and three-term relations over $e$-vertices to delete these symbols. We do this over every $e(p)$-vertex and delete all $e(p)$-symbols. Next, we apply Lemma \ref{result for many v non square} 
over the set of all $ns$-vertices and use the two-term and three-term relations over the $ns$-vertices to delete all $ns$-symbols. At worst, we are left with a relation with only $o$- and $s$-symbols. We apply Lemma \ref{result for multiple o} over each $o$-vertex and use the two-term and three-term relations over the $o$-vertices to delete all $o$-symbols.  Finally, we remain with a relation of only $s$-symbols. We decompose over the set of all $s$-vertices. There is only a two-term relation so all the symbols vanish, giving us the desired result. 
\end{proof}

\section*{Acknowledgement}
The author would like to thank Paul Gunnells for the immense amount of help and guidance he has provided and Harris Daniels for his support. He would also like to thank Jeremy Teitelbaum for showing interest in these results and Avner Ash for providing insights about the resolutions of the Steinberg module. Finally, the author would like to thank the anonymous referee for the extremely helpful comments and feedback during the review process as well as the editors for a speedy review.

\bibliographystyle{alpha}
\bibliography{references}

\end{sloppypar}

\end{document}